\documentclass[a4paper,11pt,reqno,twoside]{amsart}

\usepackage{amssymb}
\usepackage{latexsym}
\usepackage{amsmath}
\usepackage{amscd}
\usepackage{amsbsy}
\usepackage{amsthm}
\usepackage{lmodern}
\usepackage[x11names]{xcolor}
\usepackage[utf8]{inputenc}
\usepackage[T1]{fontenc}
\usepackage{csquotes}
\usepackage[paper=a4paper,left=30mm,right=30mm,top=30mm,bottom=30mm]{geometry}

\usepackage{hyperref}
\hypersetup{colorlinks=true,urlcolor=blue,linkcolor=blue,citecolor=blue,naturalnames=true,hypertexnames=false}

\usepackage[shortlabels]{enumitem}

\theoremstyle{plain}
\newtheorem{theorem}{Theorem}[section]
\newtheorem{lemma}[theorem]{Lemma}
\newtheorem{corollary}[theorem]{Corollary}
\newtheorem{proposition}{Proposition}[section]

\theoremstyle{definition}
\newtheorem{definition}[theorem]{Definition}

\newtheorem{assumption}{Assumption}[section]

\theoremstyle{remark}
\newtheorem{remark}[theorem]{Remark}

\allowdisplaybreaks

\numberwithin{equation}{section}

\usepackage{enumitem}

\renewcommand{\epsilon}{\varepsilon}
\renewcommand{\phi}{\varphi}

\newcommand{\dd}{\mathrm{d}}

\newcommand{\E}{\mathbb E}
\renewcommand{\L}{\mathbb L}
\newcommand{\N}{\mathbb N}
\renewcommand{\P}{\mathbb P}

\newcommand{\R}{\mathbb R}
\newcommand{\Z}{\mathbb Z}

\newcommand{\Bc}{\mathcal B}
\newcommand{\Dc}{\mathcal D}

\newcommand{\Fc}{\mathcal F}
\newcommand{\Ic}{\mathcal I}

\newcommand{\Nc}{\mathcal N}
\newcommand{\Qc}{\mathcal Q}
\newcommand{\Rc}{\mathcal R}
\newcommand{\Xc}{\mathcal X}
\newcommand{\Zc}{\mathcal Z}

\usepackage{mathrsfs}

\newcommand{\As}{\mathscr A}
\newcommand{\Xs}{\mathscr X}
\newcommand{\Zs}{\mathscr Z}

\usepackage{stmaryrd} \newcommand{\llbr}{\llbracket}
\newcommand{\rrbr}{\rrbracket}

\newcommand{\argmin}{\operatornamewithlimits{argmin}}

\newcommand{\KL}{\mathrm{KL}}
\newcommand{\TV}{\mathrm{TV}}

\newcommand{\supp}{\mathrm{supp}}

\usepackage{mathtools}

\newcommand{\defeq}{\vcentcolon=}
\newcommand{\eqdef}{=\vcentcolon}

\usepackage{dsfont}
\newcommand{\1}{\mathds 1}

\newcommand{\pprime}{\prime\prime}

\newcommand{\fhat}{\widehat f}
\newcommand{\ghat}{\widehat g}
\newcommand{\Bhat}{\widehat B}
\newcommand{\Shat}{\widehat S}
\newcommand{\Bprime}{B^\prime}
\newcommand{\Sprime}{S^\prime}
\newcommand{\sprime}{s^\prime}
\newcommand{\ftilde}{\widetilde f}

\newcommand{\thetatilde}{\widetilde \theta}
\newcommand{\thetacheck}{\check \theta}
\newcommand{\alphahat}{\widehat \alpha}
\newcommand{\alphaprime}{\alpha^\prime}
\newcommand{\betahat}{\widehat \beta}
\newcommand{\betatilde}{\widetilde \beta}
\newcommand{\betaprime}{\beta^\prime}

\newcommand{\Besov}[3]{\Bc^{#1}_{#2#3}}
\newcommand{\DBesov}[3]{\Dc^{#1}_{#2#3}}
\newcommand{\uppreslev}{j_1} \newcommand{\lowreslev}{j_0} \newcommand{\fzero}{f_0}

\newcommand{\threshold}{t_{j,n,\privpar}} 

\newcommand{\privpar}{\alpha} 

\newcommand{\dhamming}{d_H}

\newcommand{\sigmatilde}{\widetilde \sigma}

\renewcommand{\leqslant}{\leq}
\renewcommand{\geqslant}{\geq}

\newcommand{\fhatlin}{\fhat_{\mathrm{lin}}}
\newcommand{\festnl}{\widetilde f_n}

\newcommand{\riskp}{\Rc^{\ast}_{n, \privpar}}
\newcommand{\risk}{\Rc_n}

\newcommand{\ratep}{\mathfrak r^{\ast}_{n, \privpar}}    \newcommand{\rate}{\mathfrak r_n}    
\usepackage[backend=biber, giveninits, style=alphabetic, uniquelist=false, natbib=true, doi=false, url=false, isbn = false]{biblatex}
\DeclareFieldFormat[article]{title}{#1} \DeclareFieldFormat[book]{title}{\itshape #1}
\DeclareFieldFormat[thesis]{title}{#1}
\DeclareFieldFormat[incollection]{title}{\itshape #1}
\DeclareFieldFormat[inproceedings]{title}{\itshape #1}
\DeclareFieldFormat[article]{author}{\color{blue} #1}
\DeclareFieldFormat[article]{pages}{#1.}
\DeclareFieldFormat[incollection]{pages}{#1.} 

\renewbibmacro{in:}{\ifentrytype{article}{}{\printtext{\bibstring{in}\intitlepunct}}}

\DeclareNameAlias{sortname}{last-first}
\DeclareNameAlias{default}{last-first}

\renewbibmacro*{volume+number+eid}{\printfield{volume}\setunit*{\addnbthinspace}\printfield{number}\setunit{\addcomma\space}\printfield{eid}}
\DeclareFieldFormat[article]{number}{\mkbibparens{#1}}

\renewbibmacro*{publisher+location+date}{\printlist{publisher}\iflistundef{location}
	{\setunit*{\addcomma\space}}
	{\setunit*{\addcomma\space}}\printlist{location}\setunit*{\addcomma\space}\usebibmacro{date}\newunit}

 \DeclareFieldFormat[article]{number}{}

\title[Density estimation under local differential privacy]{Local differential privacy: Elbow effect in optimal density estimation and adaptation over Besov ellipsoids}

\date{\today}

\author[C. Butucea]{Cristina Butucea}
\address{Cristina Butucea, CREST, ENSAE, Institut Polytechnique de Paris, 5 avenue Henry Le Chatelier, F-91120 Palaiseau}
\email{cristina.butucea@ensae.fr}

\author[A. Dubois]{Amandine Dubois}
\address{Amandine Dubois, CREST-ENSAI, Campus de Ker-Lann - Rue Blaise Pascal - BP 37203 - 35172 BRUZ cedex}
\email{amandine.dubois@ensai.fr}

\author[M. Kroll]{Martin Kroll}
\address{Martin Kroll, CREST, ENSAE, Institut Polytechnique de Paris, 5 avenue Henry Le Chatelier, F-91120 Palaiseau}
\email{martin.kroll@ensae.fr}

\author[A. Saumard]{Adrien Saumard}
\address{Adrien Saumard, CREST-ENSAI, Campus de Ker-Lann - Rue Blaise Pascal - BP 37203 - 35172 BRUZ cedex}
\email{adrien.saumard@ensai.fr}


\begin{document}

\begin{abstract}
We address the problem of non-parametric density estimation under the additional constraint that only privatised data are allowed to be published and available for inference.
For this purpose, we adopt a recent generalisation of classical minimax theory to the framework of local $\privpar$-differential privacy and provide a lower bound on the rate of convergence over Besov spaces $\Besov{s}{p}{q}$ under mean integrated $\L^r$-risk.
This lower bound is deteriorated compared to the standard setup without privacy, and reveals a twofold elbow effect.
In order to fulfil the privacy requirement, we suggest adding suitably scaled Laplace noise to empirical wavelet coefficients.
Upper bounds within (at most) a logarithmic factor are derived under the assumption that $\alpha$ stays bounded as $n$ increases:
A linear but non-adaptive wavelet estimator is shown to attain the lower bound whenever $p \geq r$ but provides a slower rate of convergence otherwise.
An adaptive non-linear wavelet estimator with appropriately chosen smoothing parameters and thresholding is shown to attain the lower bound within a logarithmic factor for all cases.
\end{abstract} 
\maketitle

\section{Introduction}

\subsection*{Problem statement}

In the modern information age, increasingly more institutions are collecting and storing data.
Provided that a certain amount of privacy is guaranteed, some of these institutions might be willing to provide access to selected data sets.
Examples of such data may include information about participants in a medical study, clients of a web service, or persons interviewed in a scientific survey.
In this framework, the following questions arise naturally:
How can data be sufficiently anonymised, given a rigorous definition of privacy, and what are the consequences for subsequent data analyses resulting from the chosen anonymisation procedure?
The answer to these questions depends on several interacting parameters, namely the privacy definition at hand, the potential extent of collaboration of the involved data holding entities, and the kind of data mining tasks that should be feasible based on the private data.

In this paper, we consider the problem of non-parametric density estimation under local differential privacy as a special instance of the general problem sketched in the previous paragraph:
For $i=1,\ldots,n$, the $i$-th data holder observes a real-valued random variable $X_i$ distributed according to a probability density function $f$.
The aim is that every data holder releases an anonymised view $Z_i$ of $X_i$ such that the privacy notion of \emph{local differential privacy}, that is introduced next, is satisfied and that the density $f$ can be estimated from the data $Z_1,\ldots,Z_n$ in an optimal way.

\subsection*{Local differential private estimation}
The notion of \emph{local differential privacy} aggregates two different concepts, namely \emph{local} privacy and \emph{differential} privacy, that we explain in the sequel.

The qualitative notion of \emph{local} privacy characterises how the different entities holding the data $X_1,\ldots,X_n$ might interact to generate a private release $Z$.
It is opposed to the concept of \emph{global} privacy where the respective data holders share confidence in a common curator who has access to the ensemble of non-masked data $X_1,\ldots,X_n$ and generates the releasable data from this complete information.
In the \emph{local} setup, such an authority that is trusted by all the parties, does not exist.
However, some amount of interaction between the different parties is still allowed.
The releasable data $Z_1,\ldots,Z_n$ are obtained by successively applying suitable Markov kernels.
Given $X_i=x_i$ and $Z_1=z_1,\ldots,Z_{i-1}=z_{i-1}$, the $i$-th dataholder draws
\begin{equation*}
    Z_i \sim Q_i(\cdot \mid X_i=x_i, Z_1=z_1,\ldots,Z_{i-1}=z_{i-1})
\end{equation*}
for some Markov kernel $Q_i: \Zs \times \Xc \times \Zc^{i-1} \to [0,1]$ where the measure spaces of the non-private and private data are denoted with $(\Xc,\Xs)$ and $(\Zc,\Zs)$, respectively.
An important special case is that of \emph{non-interactive} local privacy where the random value of $Z_i$ depends on $X_i$ only and must not depend on preceding values of $Z$.
More precisely, in the non-interactive case we have
\[ Z_i \sim Q(\cdot \mid X_i = x_i) \] 
for some Markov kernel $Q$ that does no longer depend on the index $i$.
The non-interactive scenario seems to be more attractive in practice since no communication between the data holders is assumed and it is balanced in the sense that no participant obtains any information about any other participant's data.
From a mathematical point of view, however, allowing also non-interactive procedures does not lead to more technical proofs.
Thus, we potentially allow non-interactive methods in our minimax analysis, although the anonymisation techniques proposed in this paper are exclusively non-interactive.
Let us mention that for some tasks, however, interactive mechanisms provide natural and attractive alternatives (for instance, for private estimation in generalized linear models; see \cite{duchi2018minimax}, Section~5.2.1).

The notion of \emph{differential} privacy is a quantitative one and introduces a condition that makes the problem at hand mathematically tractable.
We provide its definition for the locally private case only and refer the reader to \citep{wasserman2010statistical} for a definition in the global case.

\begin{definition}A sequence of Markov kernels $Q_i: \Zs \times \Xc \times \Zc^{i-1} \to [0,1]$ provides $\privpar$-differential privacy if
\begin{equation*}
    \sup_{A \in \Zs} \frac{Q_i(A \mid X_i=x,Z_1=z_1,\ldots,Z_{i-1}=z_{i-1})}{Q_i(A \mid X_i=x^\prime,Z_1=z_1,\ldots,Z_{i-1}=z_{i-1})} \leq \exp(\privpar) \quad \text{for all }x, x^\prime \in \Xc.
\end{equation*}
In the non-interactive case, this condition is replaced with
\begin{equation*}
    \sup_{A \in \Zs} \frac{Q(A \mid X_i=x)}{Q(A \mid X_i=x^\prime)} \leq \exp(\privpar) \quad \text{for all }x, x^\prime \in \Xc.
\end{equation*}
We denote with $\Qc_\privpar$ the set of all local $\privpar$-differential private Markov kernels.
\end{definition}
Thus, the parameter $\privpar$ quantifies the amount of privacy that is guaranteed:
Setting $\privpar = 0$ ensures perfect privacy whereas letting $\privpar$ tend to infinity softens the privacy restriction.
In the non-interactive case, let us suppose that the Markov kernel $Q: \Zs \times \Xc \to [0,1]$ has a density $q$ with respect to some dominating measure.
Then, the defining property of $\privpar$-differential privacy is equivalent to
\begin{equation}\label{EQ:DEF:DIFF:PRIV:DENSITY}
    \sup_{z \in \Zc} \frac{q(z \mid X_i=x)}{q(z \mid X_i = x^\prime)} \leq \exp(\privpar) \quad \text{ for all }x,x^\prime \in \Xc.
\end{equation}

A consequence from the definition of $\privpar$-differential privacy is \emph{plausible deniability} of the data in the following sense:
Given the private view $Z_i$ only, the power of any test of the null hypothesis $H_0: X_i = x$ against the alternative $H_1: X_i = x^\prime$ with prescribed first error probability $\gamma$ has power bounded from above by $\gamma \exp(\alpha)$ (see \citep{wasserman2010statistical}, Theorem~2.4).

\subsection*{Rate optimal density estimation over Besov ellipsoids}

Let us briefly review some well-known results on non-parametric density estimation in the \emph{non-private} setup where $X_1,\ldots,X_n$ can be observed.
This classical model provides a natural benchmark for the model where additional privacy restrictions are imposed, and having in mind the results for this benchmark model turns out to be useful for understanding the ones for the model with privacy.

Density estimation from a sample $X_1,\ldots,X_n$ of observations is one of the paradigmatic problems in non-parametric statistics.
A popular framework is that of minimax optimal estimation:
Given a loss function $\ell$ (that is, a function mapping a pair of density functions $(f,g)$ to some non-negative real number) and any class $\Fc$ of candidate density functions, the quantity of interest is the minimax risk
\begin{equation}\label{EQ:DEF:MINIMAX:RISK}
    \risk(\ell,\Fc) = \inf_{\ftilde} \sup_{f \in \Fc} \E [\ell(\ftilde, f)]
\end{equation}
where the infimum is taken over all estimators (that is, $\sigma(X_1,\ldots,X_n)$-measurable functions).
In this setup, an estimator $\fhat$ is called \emph{rate optimal} if
\begin{equation*}
    \sup_{f \in \Fc} \E [\ell(\fhat, f)] \leq C(\ell,\Fc) \risk(\ell,\Fc).
\end{equation*}
Several function classes, loss functions and types of estimators have been intensively studied for the density estimation problem (see \citep{tsybakov2009introduction} and \citep{ginenickl} for comprehensive overviews of the topic). 
Throughout this paper, we consider the integrated risk associated to  $\L^r$-loss defined by $\ell(f,g) = \lVert f-g\rVert_r^r$ for $r \geq 1$.
For the Besov spaces to be considered in the sequel, wavelet methods have turned out particularly convenient.
Given a father wavelet $\phi$ and a mother wavelet $\psi$ associated to it, verifying some sufficient conditions (see conditions (5.10)--(5.12) in \citep{haerdle1998wavelets}), and an integer $\lowreslev \in \Z$, a wavelet basis of $\L^2(\R)$ is given by
\begin{equation}\label{EQ:WAVELET:BASIS}
    \{ \phi_{\lowreslev k} = 2^{\lowreslev/2} \phi(2^{\lowreslev} (\cdot) -k) \colon k \in \Z \} \cup \{ \psi_{jk} = 2^{j/2}\psi(2^j (\cdot) - k) \colon j \geq \lowreslev, k \in \Z \}.
\end{equation}
Given such a basis, the probability density $f$ admits the following formal expansion (in $\mathbb{L}^2$ sense):
\begin{equation}\label{EQ:WAV:EXP}
    f = \sum_{k \in \Z} \alpha_{j_0k} \phi_{\lowreslev k} + \sum_{j \geq \lowreslev} \sum_{k \in \Z} \beta_{jk} \psi_{jk} 
\end{equation}
where the wavelet coefficients are defined as
\begin{equation*}
    \alpha_{j_0k} = \int_\R f(x) \phi_{\lowreslev k} (x)\dd x \qquad \text{and} \qquad \beta_{jk} = \int_\R f(x) \psi_{jk}(x) \dd x.
\end{equation*}
An attractive property of wavelet expansions as \eqref{EQ:WAV:EXP} is that the membership of Besov spaces can be characterised in terms of its wavelet coefficients with respect to a well chosen wavelet basis.
In the sequel, we will work under the following assumption on the father wavelet $\phi$.

\begin{assumption}\label{COND:wavelets}
    Following \citep{haerdle1998wavelets}, we assume that the father wavelet function $\varphi$ generates a multiresolution analysis of $\mathbb{L}^2(\mathbb{R})$, that it is $N+1$ times weakly differentiable for some integer $N$ such that $0<s<N+1$, and that its derivative satisfies $\sup_x \sum_k \lvert \varphi^{(N+1)}(x-k)\rvert<\infty$ a.e.
    Moreover, we assume that there exists a bounded, non-increasing function $\Phi$ on $\mathbb{R}_+$ such that $\lvert \varphi(u)\rvert \leq \Phi(\lvert u\rvert)$ and that both $\int \Phi(\lvert u\rvert) \dd u <\infty$ and $ \int \Phi(\lvert u\rvert) \lvert u
    \rvert^N \dd u < \infty$.
\end{assumption}

If the father wavelet function $\varphi$ verifies Assumption~\ref{COND:wavelets} then, given parameters $s > 0$ and $1 \leq p,q \leq \infty$, the fact that $f$ belongs to the Besov space $\Besov{s}{p}{q}$ is equivalent to $J_{spq}(f)<\infty$ where
\begin{equation*}
    J_{spq} \defeq \lVert \alpha_{0\cdot} \rVert_{p} + \bigg( \sum_{j \geq 0} ( 2^{j(s+ 1/2 - 1/p)} \lVert \beta_{j\cdot} \rVert_p)^q \bigg)^{1/q}
\end{equation*}
for $1\leq q < \infty$ and the usual modification if $q=\infty$.
Fixing such a wavelet basis, we consider Besov ellipsoids defined as
\begin{equation*}
  \Besov{s}{p}{q}(L) = \{ f\colon \mathbb{R} \rightarrow \mathbb{R} : J_{spq}(f) \leq L \}.
\end{equation*}
Since our interest is in density estimation, a quite natural class to consider is
\begin{equation*}
    \DBesov{s}{p}{q} = \DBesov{s}{p}{q}(L,T) = \{ f : f \in \Besov{s}{p}{q}(L), f \geq 0, \int_\R f(x)\dd x = 1 \text{ and } \supp(f) \subseteq [-T,T] \},
\end{equation*}
where $\supp(f)$ denotes the support of the function $f$.
Note that we consider here the Besov smoothness of $f$ as a function defined on the whole real line, or, equivalently, that $f$ belongs to a periodic Besov class.
It would equally be possible to define Besov smoothness over the support $[-T,T]$.
Then the wavelet basis has to be boundary corrected so that it detects the smoothness on this interval only and not the potential lack of smoothness of $f$ at its boundary.
We refer the reader to \citep{ginenickl} for boundary corrected wavelets, that also dispose of all the properties that we need in the sequel.

It is well-known \citep{ginenickl,haerdle1998wavelets,donoho1996density} that
\begin{equation}\label{EQ:RATES:NONPRIVATE}
\risk(\lVert \cdot \rVert_r^r,\DBesov{s}{p}{q}) \gtrsim \rate, \text{ where }\rate =  \begin{cases} n^{-\frac{rs}{2s+1}}, & \text{ if } p > \frac{r}{2s+1} , \\ (n/\log n)^{-\frac{r(s-1/p+1/r)}{2(s-1/p)+1}}, & \text{ if } p \leq \frac{r}{2s+1},\end{cases}
\end{equation}
and these rates are optimal or suboptimal by a logarithmic factor only (see \citep{haerdle1998wavelets} for an extensive discussion).
The structural change of the rate between \emph{dense zone} (where $ p > r/(2s+1)$) and \emph{sparse zone} (where $p \leq r/(2s+1)$) is sometimes called an \emph{elbow effect}.

Moreover, in the dense case, we can distinguish the \emph{homogeneous} zone when $p \geq r$ and the \emph{non-homogeneous} zone where $r/(2s+1)<p<r$.
In the homogeneous case, \emph{linear} wavelet estimators of the form
\[ \sum_{k \in \Z} \alphaprime_{\lowreslev k} \phi_{\lowreslev k} (x) + \sum_{j = \lowreslev}^{\uppreslev} \sum_{k \in Z} \betaprime_{jk} \psi_{jk}(x) \]
with $\alphaprime_{\lowreslev k} = \frac{1}{n} \sum_{i=1}^n \phi_{\lowreslev k}(X_i)$, $\betaprime_{jk} = \frac{1}{n} \sum_{i=1}^n \psi_{j k}(X_i)$, and appropriately chosen $\lowreslev, \uppreslev$ are rate optimal whereas linear procedures are necessarily sub-optimal in the non-homogeneous case (see \citep{haerdle1998wavelets} and references therein).
In this latter scenario as well as in the sparse case, non-linear estimators based on wavelet thresholding turn out to be optimal at least up to logarithmic factors.

\subsection*{Minimax framework under privacy constraints}
Let us now describe how to extend the classical minimax setup in order to encompass the framework of local differential privacy.
Since not only the estimation procedure but also the Markov kernels guaranteeing local $\privpar$-differential privacy can freely be chosen, it is natural to replace~\eqref{EQ:DEF:MINIMAX:RISK} with the local $\privpar$-differential minimax risk defined as
\begin{equation*}
    \riskp(\ell, \Fc) = \inf_{\substack{\ftilde \\ Q \in \Qc_\privpar}} \sup_{f \in \Fc} \E_{f,Q} [\ell(f,\ftilde)].
\end{equation*}
Here the infimum is taken both over all $(\Zc,\Zs)$-measurable estimators of $f$ and all Markov kernels guaranteeing local $\privpar$-differential privacy.
A tuple $(\widehat Q, \fhat)$ consisting of a privacy mechanism and an estimator $\fhat$ is rate optimal (with respect to the local $\privpar$-differential private risk) if
\begin{equation*}
    \sup_{f \in \Fc} \E_{f,\widehat Q} \, [\ell(f,\fhat)] \leq C(\ell, \Fc) \riskp(\ell, \Fc).
\end{equation*}
The quantity $\riskp(\ell, \Fc)$ as well as the construction of optimal privacy mechanism and estimators represent the principal interest of the rest of the paper.
 
\subsection*{Related work}
Research on statistical estimation under privacy constraints is rather recent.
A landmark paper is \citep{wasserman2010statistical} where research on the subject has been initiated and density estimation via histograms and orthogonal series in the global privacy setup have been discussed.
In the same global framework, the article \citep{hall2013differential} considers anonymization of functional data and discusses kernel density estimators as the main example.
Local $\privpar$-differential privacy was intensively studied in \citep{duchi2013local} and the companion article \citep{duchi2018minimax}.
In \citep{duchi2013local} the authors show that the well-known technique of randomized response from survey statistics can be interpreted under the umbrella of local $\privpar$-differential privacy.
In the context of density estimation, \citep{duchi2013local} established minimax rates of convergence for the mean integrated squared error over Sobolev classes with arbitrary smoothness parameter $\beta \geq 1$.
They establish the minimax rate of order $n^{-\beta/(\beta +1)}$ for the mean integrated squared error over Sobolev classes with $\beta = 1$ and show that this optimal rate can be attained by Laplace perturbation of empirical histogram coefficients.
The papers \citep{duchi2013local,duchi2018minimax} provide also results for Sobolev classes with higher degrees of smoothness ($\beta > 1$) but in this case a mere perturbation of the empirical Fourier coefficients does not lead to a rate optimal method (see \citep{duchi2013local}, Observation~1 for the non-optimality of this approach).
By means of a more sophisticated sampling technique (see \citep{duchi2013local}, p.~11 or \cite{duchi2018minimax}, Section~5.2.2), however, the authors derive the minimax rate of convergence that is $(n\privpar^2)^{-\beta/(\beta + 1)}$ also in the general case.
Furthermore, \citep{duchi2013local} provides private versions of classical information-theoretical bounds that allow to apply standard lower bound techniques also in the private setup.
In~\citep{rohde2018geometrizing}, the estimation of linear functionals in the framework of local privacy is considered and a characterisation of the rates of convergence in terms of moduli of continuity is obtained which is in parallel to well-known results for the non-private setup~\citep{donoho1991geometrizing}.
This general analysis contains the private estimation of a probability density at a fixed point under mean squared error as a special case.

\bigskip

\subsection*{Main results}
In Section~\ref{SEC:LOWER}, in addition and in formal analogy to~\eqref{EQ:RATES:NONPRIVATE}, we derive, under similar technical assumptions, the following lower bound on the private minimax risk:
\begin{equation}\label{EQ:LOWER:PRIVATE}
\riskp(\lVert \cdot \rVert_r^r, \DBesov{s}{p}{q}) \gtrsim \ratep, \text{ where } \ratep = \begin{cases} (n (e^\privpar - 1)^2)^{-\frac{rs}{2s+2}}, & \text{if } p > \frac{r}{s+1}, \\ \left( \frac{n (e^\privpar - 1)^2}{\log (n (e^\privpar - 1)^2)} \right)^{-\frac{r(s-1/p+1/r)}{2(s-1/p)+2}}, & \text{if } p \leq \frac{r}{s+1}.\end{cases}
\end{equation}
This lower bound is complemented by corresponding upper bound results:
The ano\-nym\-isation technique used to create the private views of the non-releasable data $X_1,\ldots,X_n$ consists in an appropriately scaled version of the classical Laplace mechanism applied on the empirical wavelet coefficients (Section~\ref{SEC:PRIVACY}).
The wavelet estimators considered in Sections~\ref{SEC:UPPER:LINEAR} and \ref{SEC:UPPER:NONLINEAR} are based on the availability of the privatised data $Z_1,\ldots,Z_n$ only.
As in the non-private case, a linear wavelet estimator attains the given rate in the homogeneous case, that is, whenever $p \geq r$ (Section~\ref{SEC:UPPER:LINEAR}).
In Section~\ref{SEC:UPPER:NONLINEAR}, we study non-linear estimators and show that an estimator using hard thresholding can nearly attain the lower bounds both in the dense and in the sparse zone.

\subsection*{Notational conventions}

For real numbers $a,b$ we write $\llbr a,b\rrbr = [a,b] \cap \Z$.
We denote with $C$ a generic constant that might change with every appearance.
For two sequences $\{a_n\}_n,\, \{b_n\}_n$, we denote by $a_n \lesssim b_n$ that there exist some constant $C>0$ and a fixed integer number $N$ such that $a_n \leq C b_n$, for all $n\geq N$.
We say that $a_n \asymp b_n$, if both $a_n \lesssim b_n$ and $b_n \lesssim a_n$. 
If $b_n >0$, we denote by $a_n \simeq b_n$ the fact that $a_n/b_n \to 1$ as $n\to \infty$.
We recall that a centred Laplace distribution with parameter $\lambda > 0$ has the probability density function $p_\lambda(x)= \frac{1}{2\lambda} \exp(-\frac {\lvert x \rvert} {\lambda})$, for all real number $x$. In particular, if $X \sim p_\lambda$, then $\E \lvert X \rvert^k = k! \lambda^k$ for all $k \in \N$.

 \section{Lower bounds}\label{SEC:LOWER}

The purpose of this section is to derive \eqref{EQ:LOWER:PRIVATE} and hence providing an analogue of \eqref{EQ:RATES:NONPRIVATE} under local $\privpar$-differential privacy.
To this purpose, we proceed in two steps.
The first lower bound, given in Proposition~\ref{PROP:LOWER:REGULAR}, is stronger in the private dense zone ($p>r/(s+1)$), whereas the second one, given in Proposition~\ref{PROP:LOWER:SPARSE}, dominates in the private sparse zone where $p \leq r/(s+1)$.
An essential tool for both proofs is a strong information theoretical inequality (our Proposition~\ref{LEM:LOWER:EX:DUCHI}) proved in \citep{duchi2018minimax}, which states a bound for the Kullback-Leibler divergence between any distributions that have been processed through an arbitrary channel guaranteeing local $\privpar$-differential privacy.
We begin with the lower bound that is dominating in the dense zone.

\begin{proposition}\label{PROP:LOWER:REGULAR}
Let $\alpha \in (0,\infty)$ and let $L, T>0$.
Then,
    \begin{equation*}
        \inf_{\substack{\ftilde\\ Q \in \Qc_\alpha}} \sup_{f \in \DBesov{s}{p}{q}(L,T)} \E_{f,Q} \lVert \ftilde - f \rVert_r^r \gtrsim (n(e^\privpar-1)^2)^{- \frac{rs}{2s+2}},
    \end{equation*}
    where the infimum is taken over all estimators $\ftilde$ based on the private views $Z_1,\ldots,Z_n$ and all Markov kernels $Q \in \Qc_\alpha$ guaranteeing local $\privpar$-differential privacy.
\end{proposition}

The proof of Proposition~\ref{PROP:LOWER:REGULAR} is based on a reduction of the class $\DBesov{s}{p}{q}$ to a finite number of hypotheses indexed by the vertices of a hypercube of suitable dimension.
It is given in Section~\ref{SEC:PROOF:LOWER:REGULAR} in the appendix.

The following proposition complements Proposition~\ref{PROP:LOWER:REGULAR} in stating a lower bound that is stronger in the private sparse zone.

\begin{proposition}\label{PROP:LOWER:SPARSE}
Let $\alpha \in (0,\infty)$.
Let $p\geq 1$, $s \geq 1/p$ and let $L, T>0$.
Then,
    \begin{equation*}
        \inf_{\substack{\ftilde\\ Q \in \Qc_\alpha}} \sup_{f \in \DBesov{s}{p}{q}(L,T)} \E_{f,Q} \Vert \ftilde - f \Vert_r^r \gtrsim \bigg( \frac{\log (n(e^\privpar-1)^2)}{n(e^\privpar-1)^2} \bigg)^{r \cdot \frac{s-1/p+1/r}{2(s-1/p) + 2}},
    \end{equation*}
    where the infimum is taken over all estimators $\ftilde$ based on the private views $Z_1,\ldots,Z_n$ and all channels $Q \in \Qc_\alpha$ providing local $\alpha$-differential privacy.
\end{proposition}

The proof of Proposition~\ref{PROP:LOWER:SPARSE} is given in Section~\ref{SEC:PROOF:LOWER:SPARSE} in the appendix.

Taking the maximum of the lower bounds obtained in Propositions~\ref{PROP:LOWER:REGULAR} and \ref{PROP:LOWER:SPARSE} yields \eqref{EQ:LOWER:PRIVATE}.
In addition to our novel lower bounds, the known bounds \eqref{EQ:RATES:NONPRIVATE} from the non-private framework still hold true under local $\privpar$-differential privacy since processing the original data $X_1,\ldots,X_n$ through a privacy mechanism can be interpreted equivalently as imposing a restriction on the set of admissible estimators in~\eqref{EQ:DEF:MINIMAX:RISK}.
More precisely, the constraint of local $\privpar$-differential privacy confines the set of potential estimators to those of the form $\ftilde = f \circ Q$ where $Q \in \Qc_\privpar$ and $f$ is any measurable function.
Thus, 
\[ \riskp \geq \risk \vee \ratep \geq \rate \vee \ratep,\]
where the quantity $\rate$ is defined in \eqref{EQ:RATES:NONPRIVATE}.
Hence, the following corollary holds.

\begin{corollary}
    Let the assumptions of Propositions~\ref{PROP:LOWER:REGULAR} and \ref{PROP:LOWER:SPARSE} hold true. Then,
    \begin{equation*}
\riskp(\lVert \cdot \rVert_r^r, \DBesov{s}{p}{q}) \gtrsim  \begin{cases} n^{-\frac{rs}{2s+1}} \vee  (n(e^\privpar-1)^2)^{-\frac{rs}{2s+2}}, & \text{if } p > \frac{r}{s+1}, \\ n^{-\frac{rs}{2s+1}} \vee \left( \frac{n(e^\privpar-1)^2}{\log (n(e^\privpar-1)^2)} \right)^{- \frac{r(s-1/p+1/r)}{2(s-1/p)+2}}, & \text{if } \frac{r}{2s+1} < p \leq \frac{r}{s+1},\\
\left( \frac{n}{\log n} \right)^{- \frac{r(s-1/p+1/r)}{2(s-1/p)+1}}  \vee  \left( \frac{n(e^\privpar-1)^2}{\log (n (e^\privpar-1)^2)} \right)^{- \frac{r(s-1/p+1/r)}{2(s-1/p)+2}}, & \text{if }p \leq \frac{r}{2s+1}. \end{cases}
\end{equation*}
\end{corollary}

Note that the frontier between the dense and the sparse zone in the private framework is different from the one in the non-private framework leading to a partition into three regimes for the lower bound and a twofold elbow effect.
Note that these lower bounds match the upper bounds derived in Section~\ref{SEC:UPPER:LINEAR} and \ref{SEC:UPPER:NONLINEAR} at most up to logarithmic factors whenever $\alpha$ stays bounded as $n$ increases.
In addition, the bounds from the non-private setup dominate provided that $\alpha$ increases sufficiently fast in terms of $n$.
 \section{Privacy mechanisms}\label{SEC:PRIVACY}

Let us denote with $X_1,\ldots,X_n$ the real-valued random variables that represent the non-private observations held by the different data holders.
We assume that $X_1,\ldots,X_n \sim f$ for $f \in \DBesov{s}{p}{q}$.
In particular, the support of the density $f$ is contained in the interval $[-T,T]$.
In this section, we introduce a non-interactive privacy mechanism creating a private release $Z_1,\ldots,Z_n$ based on the non-private sample that satisfies the defining property of $\privpar$-differential privacy.
For this purpose, we consider a wavelet basis as in \eqref{EQ:WAVELET:BASIS}.
We assume in the sequel that the following condition on the parent wavelets is satisfied:
\begin{equation}\label{EQ:W1}
    \phi \text{ and } \psi \text{ are compactly supported on an interval } [-A,A].\tag{W1}
\end{equation}
The idea of the proposed anonymisation technique is to mask the empirical wavelet coefficients $\alphaprime_{\lowreslev k}$ and $\betaprime_{jk}$ for certain values of $j$.
A consequence of \eqref{EQ:W1} and the compact support of $f$ is that for any $\lowreslev \in \Z$ and any fixed resolution level $j \in \Z$, the corresponding $\alpha_{\lowreslev k}$ and $\beta_{jk}$ can \emph{a priori} be non-zero for a finite number of $k$ only.
We denote the set of $k$ with potentially non-zero $\alpha_{\lowreslev k}$ by $\Nc_{\lowreslev-1}$.
Analogously, for $j \geq \lowreslev$, the set of $k$ with potentially non-zero $\beta_{j k}$ is denoted with $\Nc_{j}$.

Let us now define two privacy mechanisms that will turn out to be convenient for the purposes of this paper.
It will be sufficient to consider $j_0, \, j_1 \in\mathbb{N}$ from now on.

\subsubsection*{First privacy mechanism}

For $i \in \llbr 1,n\rrbr$, $j \in \llbr \lowreslev-1, \uppreslev \rrbr$, define
\begin{equation}\label{EQ:DEF:tildeZIJK}
    Z_{ijk} = \begin{cases} \phi_{\lowreslev k}(X_i) + \sigma_{\lowreslev - 1} W_{i,\lowreslev-1,k}, & \text{ if }j= \lowreslev - 1, k \in \Nc_{\lowreslev - 1},\\
    \psi_{jk}(X_i) + \sigmatilde_j W_{ijk}, & \text{ if } j \in \llbr \lowreslev, \uppreslev \rrbr, k \in \Nc_j,
    \end{cases}
\end{equation}
where $W_{ijk}$ are independent Laplace distributed random variables with parameter $1$,
\[
\sigma_{\lowreslev - 1}= \frac{4c_A \lVert \phi  \rVert_\infty}{\alpha} \cdot 2^{\lowreslev/2} \text{ and } \sigmatilde_j = \frac{4c_A \lVert \psi  \rVert_\infty}{\alpha}\cdot \frac{\sqrt{2}}{\sqrt{2} - 1} \cdot 2^{j_1/2},
\]
for $j \in \llbr \lowreslev, \uppreslev\rrbr$ with $c_A = 2\lceil A \rceil + 1$.

\subsubsection*{Second privacy mechanism}
For $i \in \llbr 1,n\rrbr$, $j \in \llbr \lowreslev-1, \uppreslev \rrbr$, define
\begin{equation}\label{EQ:DEF:ZIJK}
    Z_{ijk} = \begin{cases} \phi_{\lowreslev k}(X_i) + \sigma_{\lowreslev - 1} W_{i,\lowreslev-1,k}, & \text{ if }j= \lowreslev - 1, k \in \Nc_{\lowreslev - 1},\\
    \psi_{jk}(X_i) + \sigma_j W_{ijk}, & \text{ if } j \in \llbr \lowreslev, \uppreslev \rrbr, k \in \Nc_j,
    \end{cases}
\end{equation}
where $W_{ijk}$ are independent Laplace distributed random variables with parameter $1$,
\[
\sigma_{\lowreslev - 1}= \frac{4c_A \lVert \phi  \rVert_\infty}{\alpha} \cdot 2^{\lowreslev/2} \text{ and } \sigma_j = \frac{4c_A \lVert \psi  \rVert_\infty}{\alpha}\cdot\frac{2\nu-1}{\nu - 1} \cdot (j \vee 1)^\nu \cdot 2^{j/2},
\]
for $j \in \llbr \lowreslev, \uppreslev\rrbr$ with $c_A = 2\lceil A \rceil + 1$ and some $\nu>1$.

Note that both privacy mechanisms in \eqref{EQ:DEF:tildeZIJK} and  \eqref{EQ:DEF:ZIJK} are non-interactive because $Z_{ijk}$ does not depend on $X_{i^\prime}$ for $i^\prime \neq i$.
The following proposition shows that both privacy mechanisms, $Z_i = (Z_{ijk})_{j \in \llbr \lowreslev-1,\uppreslev\rrbr, k \in \Nc_j}$ satisfy the condition of $\alpha$-differential privacy.

\smallskip

\begin{proposition}The privacy mechanisms given in \eqref{EQ:DEF:tildeZIJK} and \eqref{EQ:DEF:ZIJK} are local $\alpha$-differential private.
\end{proposition}

\begin{proof}
By definition of the privacy mechanism in \eqref{EQ:DEF:tildeZIJK}, the conditional density of $Z_i$ given $X_i=x_i$ can be written as
\begin{eqnarray*}
    q^{Z_i|X_i=x_i}(z_i) &=& \prod_{k\in \Nc_{\lowreslev - 1}} \frac{1}{2\sigma_{\lowreslev-1}} \exp \bigg( - \frac{\lvert z_{i,j_0-1,k} - \phi_{j_0 k}(x_i) \rvert}{\sigma_{\lowreslev - 1}} \bigg) \\
    && \cdot \prod_{j = \lowreslev}^{\uppreslev} \prod_{k\in \Nc_j} \frac{1}{2\sigmatilde_j} \exp \bigg( -  \frac{\lvert z_{ijk} - \psi_{jk}(x_i) \rvert}{\sigmatilde_j} \bigg).
\end{eqnarray*}
Thus, by the reverse and the ordinary triangle inequality,
\begin{align*}
    \frac{q^{Z_i|X_i=x_i}(z_i)}{q^{Z_i|X_i=x_i^\prime}(z_i)} &= \prod_{k \in \Nc_{\lowreslev - 1}} \exp \bigg( \frac{\lvert z_{i,j_0-1,k} - \phi_{\lowreslev k}(x_i^\prime) \vert - \vert z_{i,j_0-1, k} - \phi_{\lowreslev k}(x_i) \rvert}{\sigma_{\lowreslev-1}}  \bigg)\\
    &\hspace{2em}\cdot \prod_{j = \lowreslev}^{\uppreslev} \prod_{k \in \Nc_j} \exp \bigg(  \frac{\lvert z_{ijk} - \psi_{jk}(x_i^\prime) \vert - \vert z_{ijk} - \psi_{jk}(x_i) \rvert}{\sigmatilde_j}   \bigg)\\
    &\hspace{-4em}\leq \exp \bigg(\sum_{k \in \Nc_{\lowreslev - 1}}  \frac{\lvert \phi_{\lowreslev k}(x_i) \vert + \vert \phi_{\lowreslev k}(x_i^\prime) \rvert}{\sigma_{\lowreslev - 1}} \bigg) \cdot \exp \bigg(\sum_{j = \lowreslev}^{\uppreslev}  \sum_{k\in \Nc_j} \frac{\lvert \psi_{jk}(x_i) \vert + \vert \psi_{jk}(x_i^\prime) \rvert}{\sigmatilde_j} \bigg).
\end{align*}
Note that for any fixed $x_i$ and arbitrary $j$, $\psi_{jk}(x_i) \neq 0$ holds only for at most $c_A = 2\lceil A \rceil + 1$ different $k$, and the same argument is valid for $\phi_{\lowreslev k}(x_i)$.
Thus,
\begin{align*}
    \frac{q^{Z_i|X_i=x_i}(z_i)}{q^{Z_i|X_i=x_i^\prime}(z_i)} &\leq \exp \bigg( \frac{2\cdot 2^{\lowreslev/2}c_A \lVert \phi \rVert_\infty}{\sigma_{\lowreslev-1}} \bigg) \cdot \exp \bigg( 2 \Vert \psi \Vert_\infty c_A \cdot \sum_{j = \lowreslev}^{\uppreslev} \frac{ 2^{j/2}}{\sigmatilde_j} \bigg)\\
    &\leq \exp \bigg(\frac \alpha 2 + \frac {\alpha(\sqrt{2} - 1)}{2 \sqrt{2}} \sum_{j=j_0}^{j_1} \frac {2^{j/2}}{2^{j_1/2}} \bigg)\leq \exp(\privpar).
\end{align*}
For the privacy mechanism~\eqref{EQ:DEF:ZIJK}, analogous calculations yield for the conditional density of $Z_i$ given $X_i = x_i$ that
\begin{align*}
    \frac{q^{Z_i|X_i=x_i}(z_i)}{q^{Z_i|X_i=x_i^\prime}(z_i)} &\leq \exp \bigg( \frac{2\cdot 2^{\lowreslev/2}c_A \lVert \phi \rVert_\infty}{\sigma_{\lowreslev-1}} \bigg) \cdot \exp \bigg( 2 \Vert \psi \Vert_\infty c_A \cdot \sum_{j = \lowreslev}^{\uppreslev} \frac{ 2^{j/2}}{\sigma_j} \bigg)\\
    &\leq \exp \bigg(\frac \alpha 2 + \frac {\alpha}{2}\cdot\frac{\nu-1}{2\nu-1}(2+ \sum_{j=2}^{\infty} j^{-\nu}) \bigg)\leq \exp(\privpar),
\end{align*}
where we used that $\sum_{j=j_0}^{\uppreslev} (j\vee 1)^{-\nu} \leq \sum_{j=0}^{\infty} (j\vee 1)^{-\nu}$ and $\sum_{j=2}^{\infty} j^{-\nu} \leq (\nu - 1)^{-1}$.
\end{proof}
 \section{Upper bound for linear wavelet estimators}\label{SEC:UPPER:LINEAR}

The expansion~\eqref{EQ:WAV:EXP} suggests to consider estimators of the form
\begin{equation*}
    \fhat(x) = \sum_{k \in \Nc_{\lowreslev - 1}} \alphahat_{\lowreslev k} \phi_{\lowreslev k}(x) + \sum_{j=\lowreslev}^{\uppreslev} \sum_{k \in \Nc_j} \betahat_{jk} \psi_{jk}(x)
\end{equation*}
with appropriate estimators $\alphahat_{\lowreslev k}$ and $\betahat_{j k}$ of $\alpha_{\lowreslev k}$ and $\beta_{j k}$, respectively.
Note that in the local private framework, estimators of the wavelet coefficients are allowed to depend on the private views $Z_{ijk}$ only but not on the hidden $X_i$.
For the results concerning the linear estimator in this section, it suffices to consider the case $\lowreslev=0$.
In this case we put $\psi_{-1,k} = \phi_{0,k}$ and define a linear wavelet estimator through
\begin{align*}
    \fhatlin(x) &= \sum_{j = -1}^{\uppreslev} \sum_{k \in \Nc_j} \betahat_{jk} \psi_{jk}(x) \quad \text{with} \quad \betahat_{jk} = \frac{1}{n} \sum_{i=1}^n Z_{ijk}.
\end{align*}
Grant to $\E W_{ijk} = 0$, the definition of $\betahat_{jk}$ is natural and provides an unbiased estimate of the true wavelet coefficient $\beta_{jk}$.

The following proposition provides an upper bound for the estimator $\fhatlin$ in the so-called \emph{matched case} when $r=p$.
Its proof is given in Appendix~\ref{SEC:PROOFS:UPPER:LINEAR}.

\begin{proposition}\label{PROP:MATCHED}
Assume that the father wavelet $\varphi$ satifies Assumption~\eqref{COND:wavelets}. Let $1 \leq p < \infty$ and $Z_{ijk}$ defined as in \eqref{EQ:DEF:tildeZIJK}.
Then \begin{equation}\label{EQ:MATCHED:1}
\sup_{f \in \DBesov{s}{p}{q}} \E \lVert \fhatlin - f \rVert_p^p \lesssim 2^{-\uppreslev p s} + \bigg( \frac{2^{2\uppreslev}}{n\privpar^2} \bigg)^{p/2}  + \bigg( \frac{2^{j_1}}{n}\bigg)^{p/2}.
\end{equation}
In particular, choosing $\uppreslev = \uppreslev(n,\privpar)$ such that 
\begin{equation}\label{EQ:j1linear}
2^{\uppreslev} \asymp (n\privpar^2)^{\frac{1}{2s+2}} \wedge n^{\frac 1{2s+1}}, 
\end{equation}
we obtain
\begin{equation}\label{EQ:MATCHED:2} 
\sup_{f \in \DBesov{s}{p}{q}} \E \lVert \fhatlin - f \rVert_p^p \lesssim  (n\privpar^2)^{- \frac{ps}{2s+2}} \vee n^{-\frac {ps}{2s+1}}.
\end{equation}
\end{proposition}
The upper bound \eqref{EQ:MATCHED:2} suggests the following interpretation:
As long as $\privpar^2 \geq n^{1/(2s+1)}$, the estimator $\fhatlin$ attains the rate $n^{-ps/(2s+1)}$ known to be optimal when the sample $X_1,\ldots,X_n$ is available.
However, as soon as $\privpar^2 < n^{1/(2s+1)}$, this standard rate is deteriorated and the slower rate $(n \privpar^2)^{-ps/(2s+2)}$ is attained.
As in \citep{duchi2018minimax}, the alteration of the rate in comparison to the non-private framework concerns both the effective sample size (that changes from $n$ to $n\alpha^2$) and the exponent appearing in the rate.
In contrast to the procedure suggested in \citep{duchi2018minimax}, however, the privacy mechanism \eqref{EQ:DEF:tildeZIJK} consists in a mere perturbation of the empirical wavelet coefficients by Laplace noise, and no further sampling technique is necessary to obtain a privacy channel enabling rate optimal estimation of $f$.

Although the risk bound of Proposition~\ref{PROP:MATCHED} is valid only in the matched case, it can be extended to the case $r \neq p$ by means of the following proposition.
Its proof is given in Appendix~\ref{SEC:PROOFS:UPPER:LINEAR}.

\begin{corollary}\label{COR:UPPER:LINEAR}
Assume that the father wavelet $\varphi$ satifies Assumption~\eqref{COND:wavelets}.
Let $1 \leq p,r < \infty$ and $Z_{ijk}$ defined as in \eqref{EQ:DEF:tildeZIJK}, and put by $\sprime = s - (1/p - 1/r)_+$.
Then, choosing $\uppreslev$ as in \eqref{EQ:j1linear} yields
\begin{equation*}
    \sup_{f \in \DBesov{s}{p}{q}} \E \Vert \fhatlin - f \Vert_r^r \lesssim (n\privpar^2)^{-\frac{rs^\prime}{2\sprime+2}} \vee n^{-\frac {r\sprime}{2\sprime+1}}.
\end{equation*}
\end{corollary}

Corollary~\ref{COR:UPPER:LINEAR} together with Proposition~\ref{PROP:LOWER:REGULAR} shows that the estimator $\fhatlin$ is of optimal order in the dense homogeneous zone where $p \geq r$ (which is equivalent to $s=s^\prime$) and for $\alpha$ in $(0, \bar \alpha]$.
In analogy to \citep{donoho1996density}, it would be possible to suggest a non-linear estimation procedure depending on $s$ that is optimal (up to logarithmic factors in some cases) in the non-homogeneous dense case and in the sparse case as well.
However, in Section~\ref{SEC:UPPER:NONLINEAR}, we directly propose a non-linear estimator that is adaptive to the smoothness $s$ of the underlying density (as well as to the other parameters $p$ and $q$ of the Besov space). 
 \section{Upper bounds for the non-linear adaptive  estimator}\label{SEC:UPPER:NONLINEAR}

In this section, the privacy mechanism is given by \eqref{EQ:DEF:ZIJK} in Section~\ref{SEC:PRIVACY}. We study the theoretical properties of the \emph{non-linear} wavelet estimators of the form
\begin{equation}\label{EQ:DEF:NONLINEAR:ESTIMATOR}
    \festnl(x) = \sum_{k} \alphahat_{j_0 k} \phi_{j_0 k}(x) + \sum_{j=j_0}^{\uppreslev} \sum_k \betatilde_{jk} \psi_{jk}(x)
\end{equation}
where
\begin{equation*}
    \alphahat_{j_0 k} = \frac{1}{n} \sum_{i=1}^n Z_{i, j_0- 1,k} \quad \text{and} \quad \betatilde_{jk} = \betahat_{jk} \cdot \1_{ \{ \lvert \betahat_{jk} \rvert \geq Kt \}},
\end{equation*}
and $\betahat_{jk} = \frac{1}{n} \sum_{i=1}^n Z_{ijk}$ as in Section~\ref{SEC:UPPER:LINEAR} (the choice of $t$ and the value of the numerical constant $K$ are specified in Theorem~\ref{THM:UPPER:NONLINEAR} and its proof below).
Thus, non-linearity enters only with respect to the estimation of the detail coefficients $\beta_{jk}$.

\begin{theorem}\label{THM:UPPER:NONLINEAR}
Let the father wavelet $\varphi$ satisfy Assumption~\ref{COND:wavelets} for some integer $N>0$.
Let the private views $Z_1,\ldots,Z_n$ of the sample $X_1,\ldots,X_n$ be generated with the privacy mechanism in~\eqref{EQ:DEF:ZIJK}.
Consider the estimator $\festnl$ defined in~\eqref{EQ:DEF:NONLINEAR:ESTIMATOR} with
\begin{itemize}[itemsep=0.5em]
    \item $\lowreslev \in \N$ such that $2^{\lowreslev} \asymp (n\privpar^2)^{\frac{1}{2(N+1)+2}} \wedge n^{\frac{1}{2(N+1)+1}}$,
    \item $\uppreslev = \uppreslev^\prime \wedge \uppreslev^{\pprime}$ where $\uppreslev^{\prime}$, $\uppreslev^{\pprime} \in \N$ are such that
    \[ 2^{\uppreslev^\prime} \asymp \frac{n}{\log n} , \qquad \text{and} \qquad 2^{2\uppreslev^{\pprime}} \asymp \frac{n\privpar^2}{\log(n\privpar^2)},\]
    \item $K= 4(\overline L + \sigma)$ for some $\overline L>0$ and $ \sigma = 4 c_A \|\psi\|_\infty \cdot \frac{2\nu-1}{\nu-1}$ with $\nu$ introduced in the definition of the second privacy mechanism,
    \item $t = \threshold = \gamma \cdot \frac {j^{\nu+1/2}}{\sqrt{ n }} \cdot ( 1 \vee \frac{2^{j/2}}{\alpha})$ for $j \in \llbr \lowreslev,\uppreslev\rrbr$ and some sufficiently large constant $\gamma$ (for instance, $\gamma \geq r(N+1)$ works).
\end{itemize}
Then, the risk bound 
\begin{equation*}
     \sup_{(s,p,q,L) \in \Theta} \sup_{f \in \DBesov{s}{p}{q}(L,T)} \E \lVert \festnl - f \rVert_r^r  \lesssim (\log n)^C \cdot \mathfrak R^\star_{n,\alpha} 
     \end{equation*}
     where
     \begin{equation*}
     \mathfrak R^\star_{n,\alpha} = 
     \begin{cases} n^{-\frac{rs}{2s+1}} \vee  (n\privpar^2)^{-\frac{rs}{2s+2}}, & \text{if } p > \frac{r}{s+1}, \\ n^{-\frac{rs}{2s+1}} \vee \left( \frac{n\privpar^2}{\log (n\privpar^2)} \right)^{- \frac{r(s-1/p+1/r)}{2(s-1/p)+2}}, & \text{if } \frac{r}{2s+1} < p \leq \frac{r}{s+1},\\
\left( \frac{n}{\log n} \right)^{- \frac{r(s-1/p+1/r)}{2(s-1/p)+1}}  \vee  \left( \frac{n\privpar^2}{\log (n\privpar^2)} \right)^{- \frac{r(s-1/p+1/r)}{2(s-1/p)+2}}, & \text{if }p \leq \frac{r}{2s+1}, \end{cases}
\end{equation*}
and where
\[\Theta = (1/p, N+1) \times [1,\infty) \times [1,\infty) \times [\underline L, \overline L]\]
for some $0<\underline L \leq \overline L < \infty$.
\end{theorem}
The proof of the Theorem is given in Appendix~\ref{SEC:PROOF:UPPER:NONLINEAR}.
Note that both the privacy mechanism and the estimator in Theorem~\ref{THM:UPPER:NONLINEAR} are independent of the quantities $s$, $p$, $q$, and $L$.
Hence, the proposed procedure is adaptive. \section{Discussion}

In this article, we have suggested refined methods for density estimation under the constraint of local $\privpar$-differential privacy. 
By the use of estimators based on wavelet expansions, we have been able to obtain adaptive procedures that obtain the minimax rate of convergence up to an additional logarithmic factor only.
To the best of our knowledge, adaptation to smoothness has not been considered in the framework of private estimation so far.
Moreover, in allowing for general $\L^r$-risk and Besov ellipsoids we have widened the range of results in the privacy framework that has merely focused on $\L^2$-risk and Sobolev ellipsoids until now.

A significant difference between our approach and the one suggested in~Section~5.2.2 of \cite{duchi2018minimax} concerns the privacy mechanism:
Whereas the procedure in \cite{duchi2018minimax} is built on a rather sophisticated sampling strategy aiming at the perturbation of empirical Fourier coefficients, our privacy mechanism consists in a simple Laplace perturbation of empirical wavelet coefficients.
In \cite{duchi2018minimax} it has been observed (see the last paragraph of Section~5.2.2 in that paper) that such an approach is not feasible for the Fourier basis since it would lead to a suboptimal rate (under $\L^2$-risk) of order $(n\privpar^2)^{-2s/(2s+3)}$ over Sobolev ellipsoids of smoothness $s$ instead of the optimal rate $(n\privpar^2)^{-s/(s+1)}$.
A heuristic explanation for the easier accessibility of the problem by means of wavelet bases is given by their well-known \emph{localisation} properties in contrast to the \emph{global} Fourier basis.

Note that wavelet methods in the non-private framework do not necessarily suffer from a logarithmic loss in the rate (see, for instance, \cite{donoho1996density} where an additional logarithmic loss only appears in the dense zone).
The fact that we encounter this type of loss in our private scenario is caused be the term $j^\nu$ in the definition of the privacy mechanism~\eqref{EQ:DEF:ZIJK} and might be explained by the pointwise nature of the $\privpar$-differential privacy constraint.
The problem whether and if so, how such logarithmic losses might be circumvented remains open and provides an interesting direction for future research.

Finally, let us sketch the connection between local private estimation in the non-interactive setup and statistical inverse problems, in particular, density deconvolution:
On the one hand, in density deconvolution, the statistician is given a noisy sample $Z_1,\ldots,Z_n$ where $Z_i = X_i + \epsilon_i$ for $X_i \sim f$ and $\epsilon_i \sim q$.
Here, the density $f$ is the quantity of interest and $q$ an error density which is (at least in the overwhelming part of the literature) supposed to be known.
In this setup, the $Z_i$ are distributed according to the density $g$ where
\begin{equation}\label{EQ:DECONVOLUTION}
    g(\cdot) = (K_qf)(\cdot) \defeq \int q(\cdot - x)f(x) \dd x
\end{equation}
is the convolution of $f$ with the error density $q$.
It is well-known that the difficulty of reconstructing $f$ from the sample $Z_1,\ldots,Z_n$ is linked with the \emph{degree of ill-posedness} of the inverse problem $g=K_qf$.
The latter can be described either in terms of the sequence $(\lambda_k^2)$ of eigenvalues of $K_q^\ast K_q$ ($K_q^\ast$ denotes the adjoint operator of the linear operator $K_q$) or in terms of the decay of the Fourier transform of the error density $q$.
General inverse problems of the form $Kf=g$ have been thoroughly investigated in \cite{kerkyacharian2007needlet} in the framework of a Gaussian white noise model.
For Besov smooth signals $f$ and $\lvert \lambda_k\rvert \asymp k^{-\rho}$ for some $\rho >0$, \cite{kerkyacharian2007needlet} derived adaptive rates of estimation of $f$ proportional to

\[ \left\{
\begin{array}{ll}
    (\log n)^C n^{-\frac{rs}{2(s+\rho)+1}}, & \text{if } s > (\rho+\frac 12)(\frac{r}{p} - 1),\\
    (\log n)^C n^{- \frac{r(s-1/p+1/r)}{2(s-1/p + \rho)+1}}, & \text{if } s \leq (\rho+\frac 12)(\frac{r}{p} - 1).
\end{array}
\right.\]

On the other hand, the statistician who is given the non-interactive privatised sample $Z_1,\ldots,Z_n$ is confronted with the problem of recovering $f$ from a sample from the mixture density
\[ g(\cdot) = (Kf)(\cdot) \defeq \int q^{Z|X=x}(\cdot) f(x) \dd x, \]
which is a special instance of an inverse problem and strongly resembles \eqref{EQ:DECONVOLUTION}.
In contrast to \eqref{EQ:DECONVOLUTION}, however, the operator $K$ is now not \emph{a priori} given as a component of the problem but constitutes rather a part of its solution.
In the local differential privacy framework, the statistician should select the operator $K$, corresponding to the choice of a privacy mechanism, subject to the two following constraints.
First, the condition~\eqref{EQ:DEF:DIFF:PRIV:DENSITY} concerning $\alpha$-differential privacy must hold.
Second, the least possible amount of information should be smoothed out by the operator $K$.
More precisely, denoting with $\rho$ the degree of ill-posedness as above, the proofs of the lower bounds suggest that the least admissible value for $\rho$ is $1/2$.
Our privacy mechanisms, that is, our choices of $K$ satisfy both constraints by leading to an overall estimation procedure that is nearly minimax.
We emphasize that the above interpretation of the local differential private estimation problem does not rule out privacy mechanisms that add noise directly to the random variables $X_1,\ldots,X_n$ in principle.
As already mentioned, \cite{duchi2018minimax} have noted that adding Laplace noise directly to the observations cannot lead to an optimal procedure.
Indeed, the convolution operator in this case has degree of ill-posedness corresponding to $\rho = 2$ which yields a too slow rate.
 
\appendix

\section{Proofs of Section~\ref{SEC:LOWER}}\label{SEC:PROOFS:LOWER}

We distinguish in the sequel the dense case and the sparse case that require different explicit constructions.
However, for both proofs of the lower bounds we need the existence of a function $\fzero$ with the following properties (see \citep{haerdle1998wavelets}):
\begin{itemize}
    \item $\fzero$ is a probability density,
    \item $J_{spq}(f_0) \leq L/2$,
    \item $\supp(\fzero) \subseteq [-T,T]$,
    \item $\fzero \equiv c_0 > 0$ on some interval $[a,b]$.
\end{itemize}
In particular, $\fzero \in \DBesov{s}{p}{q}(L/2,T)$.

The main tool in the proof of the lower bounds is adapted from \citep{duchi2018minimax}. It allows to reduce the problem to the study of the likelihoods of the non-privatized data and quantifies the loss of information in the process. 

Suppose that we are given a finite indexed family of distributions $\{P_\nu, \nu\in \mathcal{V}\}$.
Let $V$ denote a random variable that is uniformly distributed over $\mathcal{V}$.
Conditionally on $V=\nu$, suppose we sample a random vector $(X_1,\ldots,X_n)$ according to the product measure $P_\nu^{\otimes n}=P_\nu\otimes\ldots\otimes P_\nu$.
Suppose that we draw an $\alpha$-locally private sample $Z_1,\ldots,Z_n$ according to a channel $Q$.
Conditioned on $V=\nu$, $(Z_1,\ldots,Z_n)$ is distributed according to the measure $M_\nu^n$ given by
$$M_\nu^n(S) \defeq \int Q^n(S\mid x_1,\ldots,x_n)\mathrm{d}P_\nu^{\otimes n}(x_1,\ldots,x_n)\quad \text{for } S\in\sigma(\Zc^n),$$
where $Q^n(\cdot\mid x_1,\ldots,x_n)$ denotes the joint distribution on $\Zc^n$ of the private sample $Z_{1:n}$ conditioned on $X_{1:n}=x_{1:n}$.
In this setup, we have the following inequality.

\begin{lemma}\label{LEM:LOWER:EX:DUCHI}[Based on \citep{duchi2018minimax}, Theorem~1] Let $\alpha \geq 0$.
For any $\alpha$-locally differentially private conditional distribution $Q$ and any $\nu,\nu'\in\mathcal{V}$, $\nu\neq \nu'$, we have in the above setting
\begin{equation*}
    \KL (M_\nu^n, M_{\nu^\prime}^n) + \KL (M_{\nu^\prime}^n, M_\nu^n) \leq 4n(e^\alpha -1)^2  \TV^2(P_\nu, P_{\nu^\prime}).
\end{equation*}
\end{lemma}

Lemma \ref{LEM:LOWER:EX:DUCHI} quantifies the property that $\alpha$-differential privacy acts as a contraction on the space of probability measures.

\subsection{Proof of Proposition~\ref{PROP:LOWER:REGULAR}}\label{SEC:PROOF:LOWER:REGULAR}

It is sufficient to prove the lower bound for $n$ sufficiently large (the remainining finitely many $n$ might merely further reduce the value of the numerical constant $C$).
Let $\fzero$ be the function introduced above.
For fixed $j$ (the choice of which will be specified later) define $\Ic_j$ as the maximal subset of $\Z$ such that $\supp(\psi_{jk}) \subset [a,b]$ and $\supp(\psi_{jk}) \cap \supp(\psi_{jk^\prime}) = \emptyset$ if $k,k^\prime \in \Ic_j$ with $k \neq k^\prime$.
Note that $N_j \defeq \vert \Ic_j \vert \asymp 2^j$.
Define
\begin{align*}
    \Fc = \{ f_\theta \colon f_\theta = \fzero + \gamma \sum_{k\in\Ic_j} \theta_k \psi_{jk} \text{ and }\theta=(\theta_k) \in \Theta \defeq \{ 0,1 \}^{N_j} \}
\end{align*}
where $\gamma = c  (n(e^\privpar-1)^2)^{-\frac{2s+1}{2(2s+2)}}$ for $c$ sufficiently small and $2^j\asymp (n(e^\privpar-1)^2)^\frac{1}{2s+2}$.
For $c$ sufficiently small, it holds $\gamma 2^{j/2} \Vert \psi \Vert_\infty\leqslant c_0$, which ensures that $f_\theta$ is non-negative for all $\theta \in \Theta$.
One can easily check that $\int f_\theta = 1$ and $\supp(f_\theta) \subseteq [-T,T]$ for all $\theta \in \Theta$.
Moreover, by the definition of $\gamma$, the choice of $j$ and the equivalence of norms, we have
\begin{align*}
    \Vert f_\theta \Vert_{spq} &\leq \Vert \fzero \Vert_{spq} + c_1\gamma 2^{j(s+1/2-1/p)} \big( \sum_{k\in\Ic_j} \vert \theta_k \vert^p \big)^{1/p}\\
    &\leq \frac{L}{2} + c_1\gamma 2^{j(s+1/2)}\leq \frac{L}{2}+Ccc_1 \leq L,
\end{align*}
where the last inequality holds for $c$ sufficiently small.
Hence, $\Fc \subset \DBesov{s}{p}{q}(L,T)$ and
\[\sup_{f \in \DBesov{s}{p}{q}(L,T)} \E_f \Vert \ftilde - f \Vert_r^r \geq \sup_{f \in \Fc} \E_f\Vert \ftilde - f \Vert_r^r = \max_{\theta\in\Theta} \E_\theta \Vert \ftilde - f_\theta \Vert_r^r .\]
Denoting by $\Delta_{jk}$ the support of $\psi_{jk}$, it holds for any estimator $\ftilde$ of $f$ that
\begin{align*}
    \E_\theta \Vert \ftilde - f_\theta \Vert_r^r &= \E_\theta \int
    \vert \ftilde(x) - f_\theta(x) \vert^r \dd x\\
    &\geq \sum_{k\in\Ic_j} \E_\theta \int_{\Delta_{jk}} \vert \ftilde(x) - f_\theta(x) \vert^r \dd x\\
    &= \sum _{k\in\Ic_j} \E_\theta \int_{\Delta_{jk}} \vert \ftilde(x) - \fzero(x) - \gamma \theta_k \psi_{jk}(x) \vert^r \dd x
\end{align*}
since $f_\theta\equiv g_{\theta_k}:=f_0+\gamma\theta_k\psi_{jk}$ on $\Delta_{jk}$.
Set $$\Vert\ftilde- g_{\theta_k}\Vert_{r,\Delta_{jk}}^r=\int_{\Delta_{jk}} \vert \ftilde(x) - g_{\theta_k}(x) \vert^r \dd x = \int_{\Delta_{jk}} \vert \ftilde(x) - \fzero(x) - \gamma \theta_k \psi_{jk}(x) \vert^r \dd x,$$ and $\thetacheck_k = \argmin_{\theta \in \{0,1\} } \Vert\ftilde- g_\theta\Vert_{r,\Delta_{jk}}$. It follows from the triangle inequality that
\begin{align*}
2\Vert\ftilde- g_{\theta_k}\Vert_{r,\Delta_{jk}}&\geq \Vert\ftilde- g_{\theta_k}\Vert_{r,\Delta_{jk}}+\Vert\ftilde- g_{\thetacheck_k}\Vert_{r,\Delta_{jk}}\\
&\geq \Vert g_{\theta_k}- g_{\thetacheck_k}\Vert_{r,\Delta_{jk}}\\
&=\gamma\vert \theta_k-\thetacheck_k\vert \Vert\psi_{jk}\Vert_r.
\end{align*}
Thus,
\begin{align*}
    \E_\theta \Vert \ftilde - f_\theta \Vert_r^r &\geq \frac{\gamma^r}{2^r} \sum_{k\in\Ic_j} \E_\theta [\vert\thetacheck_k - \theta_k\vert^r] \Vert \psi_{jk} \Vert_r^r\\
    &= \frac{\gamma^r}{2^r} \Vert \psi_{j1} \Vert_r^r \cdot \E_\theta [ \dhamming(\thetacheck, \theta)],
\end{align*}
where $\dhamming$ denotes the Hamming distance.
Therefore,
\begin{equation*}
    \sup_{f \in \DBesov{s}{p}{q}(L,T)} \E_f \Vert \ftilde - f \Vert_r^r \geq \max_{\theta\in\Theta} \E_\theta \Vert \ftilde - f_\theta \Vert_r^r \geq \frac{\gamma^r}{2^r} \Vert \psi_{j1} \Vert_r^r \cdot \inf_{\thetatilde} \max_{\theta\in\Theta} \E_\theta [ \dhamming(\thetatilde, \theta)].
\end{equation*}
In order to apply Lemma~\ref{LEM:ASSOUAD:KL}, we need to  bound the Kullback-Leibler divergence between two different distributions $M^n_\theta$ and $M_{\theta'}^n$ of the private sample $(Z_1,\ldots,Z_n)$ resulting from the sample $X_1,\ldots,X_n$ if, for all $i\in\llbr 1,n\rrbr$, $X_i$ is distributed according to $f_\theta$, $f_{\theta^\prime}$ with $\dhamming(\theta, \theta^\prime) = 1$.
We write $X_i\sim \P_\theta$ if $X_i$ has density $f_\theta$.
Using Lemma~\ref{LEM:LOWER:EX:DUCHI} we obtain for any channel providing local $\privpar$-differential privacy that
\begin{equation*}
    \KL(M_\theta^n,M_{\theta^\prime}^n) \leq 4(e^\alpha-1)^2n \TV^2(\P_\theta, \P_{\theta^\prime}).
\end{equation*}
Now, since $\dhamming(\theta, \theta^\prime)=1$ and $\theta, \theta^\prime\in\Theta$, there exists $k_0\in\Ic_j$ such that
\begin{align*}
    \TV(\P_\theta, \P_{\theta^\prime}) &= \frac{1}{2}\int \lvert f_\theta(x)-f_{\theta'}(x) \rvert \dd x= \frac{1}{2} \int \lvert\gamma \sum_{k\in \Ic_j}(\theta_k-\theta_k') \psi_{jk}(x)\rvert \dd x\\
    &=\frac{\gamma}{2} \int \lvert  \psi_{jk_0}(x)\rvert \dd x=\frac{1}{2}2^{-j/2} \gamma \lVert \psi \rVert_1,
\end{align*}
which implies that
\begin{equation*}
    \KL(M_\theta^n, M_{\theta^\prime}^n) \leq (e^\alpha-1)^2\Vert \psi\Vert_1^2n2^{-j}\gamma^2\leq c^2\Vert\psi\Vert_1^2 C <\infty.
\end{equation*}
Applying Lemma~\ref{LEM:ASSOUAD:KL} from the appendix with $N=N_j \gtrsim 2^j$  implies
\begin{align*}
    \sup_{f \in \DBesov{s}{p}{q}(L,T)} \E_f \Vert \ftilde - f \Vert_r^r &\gtrsim \frac{\gamma^r}{2^r}2^{j(r/2-1)}  \Vert \psi \Vert_r^r \cdot  2^j  \\
    &\gtrsim  (n(e^\privpar-1)^2)^{-\frac{rs}{2s+2}}.
\end{align*}
This implies the statement of the proposition since $\ftilde$ and the channel distribution were arbitrary.

\subsection{Proof of Proposition~\ref{PROP:LOWER:SPARSE}}\label{SEC:PROOF:LOWER:SPARSE}

We consider $\fzero,\psi, \Ic_j$ and $N_j$ as in the proof of Proposition~\ref{PROP:LOWER:REGULAR}, but consider now the set
\begin{equation*}
    \Fc = \{ f_k=f_0 + \gamma \cdot \psi_{jk}, k \in \Ic_j \}\cup\{f_0\},
\end{equation*}
where $j$ is chosen such that $2^j \simeq \big( \frac{n(e^\privpar-1)^2}{\log(n (e^\privpar-1)^2 )} \big)^{\frac{1}{2(s+1-1/p)}}$ and  $\gamma = c2^{-j(s+1/2-1/p)}$ for $c$ sufficiently small.
Let us first check that this choice of $j$ and $\gamma$ guarantees that $\Fc \subset \DBesov{s}{p}{q}(L,T)$.
First, we have $\fzero \in \DBesov{s}{p}{q}(L,T)$ and one can easily check that $\int f_k = 1$ and $\supp(f_k) \subseteq [-T,T]$ for all $k \in \Ic_j$.
Then, for any $k \in \Ic_j$, we have
\begin{equation*}
    f_k \geq c_0 - \gamma \Vert \psi_{jk} \Vert_\infty \geq c_0 - c2^{-j(s+1/2-1/p)} 2^{j/2} \Vert \psi \Vert_\infty  \geq  c_0 - c\Vert \psi \Vert_\infty\geq 0
\end{equation*}
for $c$ sufficiently small.
Furthermore, for any $k \in \Ic_j$,
\begin{equation*}
    \Vert f_k \Vert_{spq} \leq \Vert f_0 \Vert_{spq} + \gamma \Vert \psi_{jk} \Vert_{spq} \leq L/2 + c2^{-j(s+1/2-1/p)} \Vert \psi_{jk} \Vert_{spq} \leq L/2 + cc_1 \leq L
\end{equation*}
for $c$ sufficiently small.
Hence, $\Fc \subset \DBesov{s}{p}{q}(L,T)$ and
$$\sup_{f \in \DBesov{s}{p}{q}(L,T)} \E_f \Vert \ftilde - f \Vert_r^r \geq \sup_{f \in \Fc} \E_f\Vert \ftilde - f \Vert_r^r.$$
Now, we show that for $k,k'\in\Ic_j$, $k\neq k'$, the hypotheses $f_k$ and $f_{k^\prime}$, as well as the hypotheses $f_k$ and $f_0$, are sufficiently separated in the sense of Lemma~\ref{LEM:MULTIPLE:HYP:KL}.
For such $k, k^\prime$ we have:
\begin{align*}
    \Vert f_k - f_{k^\prime} \Vert_r^r &\geq \Vert f_k - f_{0} \Vert_r^r = \gamma^r 2^{rj(1/2-1/r)} \cdot \Vert \psi \Vert_r^r = c^r 2^{-rj(s+1/2-1/p)} 2^{rj(1/2-1/r)}\cdot \Vert \psi \Vert_r^r \\
    &= c^r \Vert \psi \Vert_r^r  2^{-jr(s-1/p+1/r)}\\
    &\geq C \bigg( \frac{\log (n (e^\privpar-1)^2)}{n (e^\privpar-1)^2} \bigg)^{r \cdot \frac{s-1/p+1/r}{2(s+1-1/p)}}.
\end{align*}
For $k\in\{0\}\cup\Ic_j$, let $M_k^n$ be the distribution of the private sample $(Z_1,\ldots,Z_n)$ resulting from the sample $X_1,\ldots,X_n$ if for all $i\in\llbr 1,n\rrbr$ $X_i$ is distributed according to $f_k$.
For all $k\in\Ic_j$ we have $M_k^n\ll M_0^n$. It remains to bound the quantity $\frac{1}{N_j}\sum_{k\in\Ic_j}\KL(M_k^n,M_0^n).$
We write $X_i\sim \P_k$ if $X_i$ has density $f_k$, $k\in\{0\}\cup\Ic_j$.
First consider  the total variation distance between $\P_k$ and $\P_0$ for $k\in\Ic_j$ :
\begin{align*}
    \TV (\P_k, \P_0) &=\frac{1}{2}\int \vert f_k-f_0 \vert=\frac{\gamma}{2}\int \vert \psi_{jk}\vert=\frac{\gamma}{2}2^{-j/2} \Vert \psi \Vert_1\\
    &= \frac{c}{2}\Vert \psi \Vert_1 2^{-j(s-1/p + 1)},
\end{align*}
and thus
\begin{equation*}
    \TV^2(\P_k, \P_0) \leq \frac{c^2}{4} \Vert \psi \Vert_1^2 C\cdot \frac{\log(n(e^\privpar-1)^2)}{n (e^\privpar-1)^2} .
\end{equation*}

Applying Lemma~\ref{LEM:LOWER:EX:DUCHI} gives
\begin{equation}\label{EQ:BOUND:KL}
    \frac{1}{N_j}\sum_{k\in\Ic_j}\KL(M_k^n, M_0^n) \leq c^2 \Vert \psi \Vert_1^2 C \cdot \log(n(e^\privpar-1)^2).
\end{equation}
Now, $\log(N_j)\geq \log(C2^j)$ and
\begin{equation*}
  \log (C2^j )>  \frac{\log (n(e^\privpar-1)^2)}{2(s-1/p+1)}(1+o(1)) \geq \frac{1}{2}\frac{\log (n(e^\privpar-1)^2)}{2(s-1/p+1)}
\end{equation*}
for $n$ sufficiently large, say $n \geq n_0$.
Putting this estimate into \eqref{EQ:BOUND:KL} yields
\begin{equation*}
    \frac{1}{N_j}\sum_{k\in\Ic_j}\KL(M_k^n, M_0^n) \leq C \log (N_j)
\end{equation*}
for $n \geq n_0$ and $C<1/8$ for $c$ sufficiently small.
We can then apply Lemma~\ref{LEM:MULTIPLE:HYP:KL}, which yields for $n \geq n_0$ that
\begin{equation*}
        \sup_{f \in \DBesov{s}{p}{q}} \E_f \Vert \ftilde - f \Vert_r^r \geq C \bigg( \frac{\log (n(e^\privpar-1)^2)}{n(e^\privpar-1)^2} \bigg)^{r \cdot \frac{s-1/p+1/r}{2(s-1/p) + 2}}.
    \end{equation*}
The statement of the proposition follows since both the estimator $\ftilde$ and the privacy mechanism considered were arbitrary.

\subsection{Further auxiliary results for the lower bound proofs}

The following lemma is a Kullback-Leibler version of Assouad's lemma.
As above, we denote by $\dhamming$ the Hamming distance, that is, $\dhamming(\theta, \theta^\prime) = \sum_{i=1}^d \1_{ \{\theta_i \neq \theta_i^\prime\}}$ for $\theta, \theta^\prime \in \R^d$.

\begin{lemma}[\citep{tsybakov2009introduction}, p.~118, Theorem~2.12]\label{LEM:ASSOUAD:KL}
Denote with $\Theta = \{ 0,1 \}^N$ the set of all binary sequences of length $N$.
Let $\{ \P_\theta : \theta \in \Theta \}$ be a set of $2^N$ probability measures on some measurable space $(\Xc,\As)$ and let the corresponding expectations be denoted by $\E_\theta$.
Then
\begin{equation*}
    \inf_{\thetatilde} \max_{\theta \in \Theta} \E_\theta [\dhamming(\theta,\thetatilde)] \geq \frac{N}{2} \max \{ \exp(-\beta)/2, 1-\sqrt{\beta/2} \}
\end{equation*}
provided that $\KL(\P_\theta, \P_{\theta^\prime}) \leq \beta < \infty$ for all $\theta, \theta^\prime \in \Theta$ with $\dhamming(\theta, \theta^\prime) = 1$.
\end{lemma}

For the lower bound in the sparse case we need the following lemma taken from~\citep{tsybakov2009introduction}.

\begin{lemma}[\citep{tsybakov2009introduction}, p.~101, Theorem~2.7]\label{LEM:MULTIPLE:HYP:KL}
Assume that $M \geq 1$ and suppose that $\Theta$ contains elements $\theta_0, \theta_1,\ldots, \theta_M$ such that:
\begin{enumerate}[label=(\roman*)]
    \item $d(\theta_j, \theta_k) \geq 2\Psi > 0$, for all $0 \leq j < k \leq M$,
    \item $\P_j \ll \P_0$, for all $j= 1, \ldots, M$, and
    \begin{equation*}
        \frac{1}{M} \sum_{j=1}^M \KL(\P_j, \P_0) \leq \beta \log M
    \end{equation*}
    with $0 < \beta < 1/8$ and $\P_j = \P_{\theta_j}$, $j = 0,1,\ldots,M$. Then
    \begin{equation*}
        \inf_{\widetilde \theta} \sup_{\theta \in \Theta} \E_\theta(d^r(\widetilde \theta, \theta)) \geq c(\beta) \Psi^r.
    \end{equation*}
\end{enumerate}
\end{lemma} \section{Proofs of Section~\ref{SEC:UPPER:LINEAR}}\label{SEC:PROOFS:UPPER:LINEAR}

\subsection{Proof of Proposition~\ref{PROP:MATCHED}}

We give the proof for $p > 2$ only, which is based on Statement~\ref{LEM:ROSENTHAL:GEQ} from Lemma~\ref{LEM:ROSENTHAL}. The proof for $1 \leq p \leq 2$ follows similarly using \ref{LEM:ROSENTHAL:LEQ} instead.
We decompose the risk of the estimator $\fhatlin$ into approximation and stochastic error:
\begin{equation*}
    \E \lVert \fhatlin - f \rVert_p^p \leq 2^{p-1} \{ \E \lVert \fhatlin - \E [\fhatlin] \rVert_p^p + \lVert \E [\fhatlin] - f \rVert_p^p \}.
\end{equation*}
The approximation term can be dealt with exactly as in the case of non-private data (see \citep{haerdle1998wavelets}, p.~130),
\begin{equation*}
    \lVert \E [\fhatlin] - f \rVert_p^p \leq C2^{-sp\uppreslev},
\end{equation*}
and it remains to consider the stochastic term.
Putting $\beta^\prime_{-1,k} = \frac{1}{n} \sum_{i=1}^n \phi(X_i-k)$ and $\beta^\prime_{jk} = \frac{1}{n} \sum_{i=1}^n \psi_{jk}(X_i)$, we have
\begin{align*}
    \fhatlin - \E [\fhatlin] &= \sum_{j = -1}^{\uppreslev} \sum_{k \in \Nc_j} \beta^\prime_{jk} \psi_{jk}(x) - \sum_{j=-1}^{\uppreslev} \sum_{k\in \Nc_j} \beta_{jk} \psi_{jk}(x)\\
    &+\sum_{k \in \Nc_{-1}} \left( \frac{1}{n} \sum_{i=1}^n \sigma_{-1} W_{i,-1,k} \right) \psi_{-1,k}(x) + \sum_{j=0}^{\uppreslev} \sum_{k\in \Nc_j} \left( \frac{1}{n} \sum_{i=1}^n \sigmatilde_j W_{ijk} \right) \psi_{jk}(x),
\end{align*}
which can be rewritten as
\begin{align*}
    \fhatlin - \E [\fhatlin] &= \frac{1}{n} \sum_{i=1}^n K_{\uppreslev + 1}(x,X_i) - \E [K_{\uppreslev + 1}(x,X_1)] \\
    &+\sum_{k\in \Nc_{-1}} \left( \frac{1}{n} \sum_{i=1}^n \sigma_{-1} W_{i,-1,k} \right) \psi_{-1,k}(x) + \sum_{j=0}^{\uppreslev} \sum_{k\in \Nc_j} \left( \frac{1}{n} \sum_{i=1}^n \sigmatilde_j W_{ijk} \right) \psi_{jk}(x),
\end{align*}
where $K_{j}(x,y) = 2^j \sum_k \phi(2^j x -k) \bar \phi(2^j y -k)$.
We further decompose
\begin{align*}
    \E \Vert \fhatlin - \E [\fhatlin] \Vert_p^p &\leq 2^{p-1} \E \bigg\Vert \frac{1}{n} \sum_{i=1}^n K_{\uppreslev + 1}(\cdot,X_i) - \E (K_{\uppreslev + 1}(\cdot,X_1)) \bigg\Vert_p^p\\
    &\hspace{-5em}+ 2^{p-1} \E \bigg\Vert \sum_{k\in \Nc_{-1}} \left( \frac{1}{n} \sum_{i=1}^n \sigma_{-1} W_{i,-1,k} \right) \psi_{-1,k} + \sum_{j=0}^{\uppreslev} \sum_{k\in \Nc_j} \left( \frac{1}{n} \sum_{i=1}^n \sigma_j W_{ijk} \right) \psi_{jk} \bigg\Vert_p^p. 
\end{align*}
The first term on the right-hand side is analysed as in the non-private setup (see~\citep{haerdle1998wavelets}, p.~130) leading to the bound
\begin{equation}\label{EQ:MATCHED:STOCH:1}
    \E \bigg\Vert \frac{1}{n} \sum_{i=1}^n K_{\uppreslev + 1}(\cdot,X_i) - \E [K_{\uppreslev + 1}(\cdot,X_1)] \bigg\Vert_p^p \leq C \left( \frac{2^{\uppreslev}}{n} \right)^{p/2}.
\end{equation}
For the remaining term, we have by Tonelli's theorem
\begin{align*}
    \E &\int \bigg\vert \sum_{k\in \Nc_{-1}} \left( \frac{1}{n} \sum_{i=1}^n \sigma_{-1} W_{i,-1,k} \right) \psi_{-1,k}(x) + \sum_{j=0}^{\uppreslev} \sum_{k\in \Nc_j} \left( \frac{1}{n} \sum_{i=1}^n \sigmatilde_j W_{ijk} \right) \psi_{jk}(x) \bigg\vert^p \dd x\\
    &= \frac{1}{n^p} \int_\Delta \E \bigg\vert \sum_{k \in \Nc_{-1}} \left( \sum_{i=1}^n \sigma_{-1} W_{i,-1,k} \right) \psi_{-1,k}(x) + \sum_{j=0}^{\uppreslev} \sum_{k\in \Nc_j} \left( \sum_{i=1}^n \sigmatilde_j W_{ijk} \right) \psi_{jk}(x) \bigg\vert^p \dd x 
\end{align*}
where $\Delta$ is some compact set the length of which depends on $A$ and $T$ only.
The expectation inside the integral is bounded from above by Rosenthal's inequality (Statement~\ref{LEM:ROSENTHAL:GEQ} of Lemma~\ref{LEM:ROSENTHAL}):
\begin{align*}
    \E &\bigg\vert \sum_{k\in \Nc_{-1}} \left( \sum_{i=1}^n \sigma_{-1} W_{i,-1,k} \right) \psi_{-1,k}(x) + \sum_{j=0}^{\uppreslev} \sum_{k\in \Nc_j} \left(  \sum_{i=1}^n \sigmatilde_j W_{ijk} \right) \psi_{jk}(x) \bigg\vert^p\\
    &\lesssim \sum_{k\in \Nc_{-1}} \sum_{i=1}^n \E \big\vert \sigma_{-1}W_{i,-1,k} \psi_{-1,k}(x) \vert^p + \sum_{j=0}^{j_1} \sum_{k \in \Nc_j} \sum_{i=1}^n \E \vert \sigmatilde_{j} W_{ijk} \psi_{jk}(x) \big\vert^p \\
    & + \bigg( \sum_{k\in \Nc_{-1}} \sum_{i=1}^n \E \vert \sigma_{-1}W_{i,-1,k} \psi_{-1,k}(x) \vert^2 + \sum_{j=0}^{j_1} \sum_{k\in \Nc_{j}} \sum_{i=1}^n \E \vert \sigmatilde_{j} W_{ijk} \psi_{jk}(x) \vert^2 \bigg)^{p/2}\\
    &= n \sum_{k\in \Nc_{-1}} \sigma_{-1}^p \vert \psi_{-1,k}(x) \vert^p \E \vert W_{1,-1,k} \vert^p + n \sum_{j=0}^{j_1} \sum_{k\in \Nc_{j}} \sigmatilde_{j}^p \vert \psi_{jk}(x)\vert^p \E \vert W_{1jk}\vert^p \\
    &+ n^{p/2} \bigg( \sum_{k \in \Nc_{-1}} \sigma_{-1}^2 \vert \psi_{-1,k}(x) \vert^2 \E \vert W_{1,-1,k} \vert^2 + \sum_{j=0}^{j_1} \sigmatilde_{j}^2 \sum_{k \in \Nc_{j}} \vert \psi_{jk}(x) \vert^2 \E \vert  W_{1jk} \vert^2 \bigg)^{p/2}\\
    &\asymp n \sum_{k \in \Nc_{-1}} \sigma_{-1}^p \vert \psi_{-1,k}(x) \vert^p + n \sum_{j=0}^{j_1} \sum_{k\in \Nc_{j}} \sigmatilde_{j}^p \vert \psi_{jk}(x)\vert^p\\
    &+ n^{p/2} \bigg( \sum_{k \in \Nc_{-1}} \sigma_{-1}^2 \vert \psi_{-1,k}(x) \vert^2 + \sum_{j=0}^{j_1} \sigmatilde_{j}^2 \sum_{k\in \Nc_{j}} \vert \psi_{jk}(x) \vert^2  \bigg)^{p/2}\\
    &\asymp n \sum_{k\in \Nc_{-1}} \vert \psi_{-1,k}(x) \vert^p \frac{1}{\alpha^p} + n \sum_{j=0}^{j_1} \sum_{k\in \Nc_{j}} 2^{j_1 p/2} \vert  \psi_{jk}(x)\vert^p \frac{1}{\alpha^p} \\
    &+ n^{p/2} \bigg( \sum_{k\in \Nc_{-1}} \vert \psi_{-1,k}(x) \vert^2 \frac{1}{\alpha^2} + \sum_{j=0}^{j_1} 2^{j_1} \sum_{k \in \Nc_{j}} \vert \psi_{jk}(x) \vert^2 \frac{1}{\alpha^2} \bigg)^{p/2}.
\end{align*}
Recall the definition of $\psi_{jk}$ and noting that grant to the boundedness of the support of the wavelet parents $\phi$ and $\psi$ we have for any $x$ and fixed $j$ that $\psi_{jk}(x) \neq 0$ only for a finite number of $k$ that is independent of $j$.
Thus, using the last expression we bound from above as follows
\begin{align*}
&\int_\Delta \E \bigg\vert \sum_{k \in \Nc_{-1}} \left( \sum_{i=1}^n \sigma_{-1} W_{i,-1,k} \right) \psi_{-1,k}(x) + \sum_{j=0}^{\uppreslev} \sum_{k\in \Nc_j} \left( \sum_{i=1}^n \sigmatilde_j W_{ijk} \right) \psi_{jk}(x) \bigg\vert^p \dd x \\
    &\leq C \bigg(\frac{n}{\alpha^p} + n 2^{j_1 p/2} \sum_{j=0}^{\uppreslev} \frac{2^{j p/2}}{\privpar^p} + n^{p/2} \bigg( \frac{1}{\privpar^2} + 2^{j_1} \sum_{j=0}^{j_1} \frac{2^j}{\alpha^2} \bigg)^{p/2}\bigg)\\
    &\simeq \frac{n}{\privpar^p} + n \cdot \frac{ 2^{p\uppreslev}}{\privpar^p} + \frac{n^{p/2}}{\privpar^p} +  n^{p/2}\cdot  \frac{ 2^{p \uppreslev}}{\privpar^p}.
\end{align*}
Thus,
\begin{align}
    \E &\int_\Delta \bigg\vert \sum_{k\in \Nc_{-1}} \bigg( \frac{1}{n} \sum_{i=1}^n \sigma_{-1} W_{i,-1,k} \bigg) \psi_{-1,k}(x) + \sum_{j=0}^{\uppreslev} \sum_{k\in \Nc_j} \left( \frac{1}{n} \sum_{i=1}^n \sigmatilde_j W_{ijk} \right) \psi_{jk}(x) \bigg\vert^p \dd x\nonumber\\
    & \lesssim \frac{ 2^{p\uppreslev}}{\alpha^p n^{p-1}} + \left( \frac{2^{2\uppreslev}}{n\alpha^2} \right)^{p/2}.\label{EQ:DEF:2ndterm}
\end{align}
Combining \eqref{EQ:MATCHED:STOCH:1} and \eqref{EQ:DEF:2ndterm} yields
\begin{equation*}
    \E \Vert \fhatlin - \E [\fhatlin] \Vert_p^p \lesssim  \left( \frac{2^{2\uppreslev}}{n\alpha^2} \right)^{p/2} + \left(\frac{2^{j_1}}{n} \right)^{p/2},
\end{equation*}
which proves \eqref{EQ:MATCHED:1}.
Choosing $\uppreslev = \uppreslev(n,\alpha)$ as in \eqref{EQ:j1linear} immediately implies \eqref{EQ:MATCHED:2}.

\subsection{Proof of Corollary~\ref{COR:UPPER:LINEAR}}

We distinguish between the cases $p \geq r$ and $p < r$.

\paragraph{\emph{1. Case}: $p > r$}
In this case, $s^\prime = s$. Let us consider the estimator $\fhatlin$ with $\uppreslev$ chosen as in Proposition~\ref{PROP:MATCHED}.
First note that there exists a constant $\bar C > 0$ such that the Lebesgue measure of $\supp(\fhatlin - f))$ is bounded from above by a constant $\bar C > 0$.
Then, applying H\"older's inequality and Proposition~\ref{PROP:MATCHED} yields
\begin{align*}
    \E \Vert \fhatlin - f \Vert_r^r \leq \bar C^{1-r/p} \big( \E \| \fhatlin - f \|_p^p \big)^{r/p} \lesssim (n\alpha^2)^{-\frac{rs}{2s+2}} \vee n^{-\frac{rs}{2s+1}}.
\end{align*}

\paragraph{\emph{2. Case}: $p \leq r$}
In this case, $\sprime = s - 1/p + 1/r$.
Grant to the Besov embedding it holds $\Besov{s}{p}{q} \subset \Besov{\sprime}{r}{q}$, which implies $\DBesov{s}{p}{q} \subset \DBesov{\sprime}{r}{q}$.
Thus, again using the upper bound for the matched case from Proposition~\ref{PROP:MATCHED}, 
\begin{align*}
    \sup_{f \in \DBesov{s}{p}{q}} \E \Vert \fhatlin - f \Vert_r^r 
    &\leq  \sup_{f \in \DBesov{s^\prime}{r}{q}} \E \Vert \fhatlin - f \Vert_r^r\\
    &\lesssim (n\alpha^2)^{-\frac{r s^\prime }{2s^\prime + 2}} \vee n^{-\frac{r s^\prime}{2 s^\prime + 2}},
\end{align*}
which is the desired bound for the case $p \leq r$.

\subsection{Inequalities for moments of sums of independent random variables}

\begin{lemma}\label{LEM:ROSENTHAL}
Let $X_1,\ldots,X_n$ be independent centred random variables with $\E [\vert X_i \vert^r] < \infty$.
\begin{enumerate}[label=(\roman*)]
    \item\label{LEM:ROSENTHAL:LEQ} If $0 < r \leq 2$, then
    \begin{equation*}
    \E \left( \bigg \lvert \sum_{i=1}^n X_i \bigg \rvert^r \right) \leq \bigg( \sum_{i=1}^n \E (X_i^2) \bigg)^{r/2}.
\end{equation*}
    \item\label{LEM:ROSENTHAL:GEQ} If $r > 2$, then there exists a constant $C=C(r)$ such that
    \begin{equation*}
    \E \left( \bigg \lvert \sum_{i=1}^n X_i \bigg\rvert^r \right) \leq C \bigg\{ \sum_{i=1}^n \E (\vert X_i \vert^r) + \bigg( \sum_{i=1}^n \E (X_i^2) \bigg)^{r/2} \bigg\}.
\end{equation*}
\end{enumerate}
\end{lemma}
Inequality~\ref{LEM:ROSENTHAL:LEQ} follows directly from Jensen's inequality and concavity of $x \mapsto x^{r/2}$ for $0 < r \leq 2$.
For a proof of inequality~\ref{LEM:ROSENTHAL:GEQ} we refer to \citep{petrov1995limit}, p.~59, Theorem~2.9. \section{Proof of Theorem~\ref{THM:UPPER:NONLINEAR}}\label{SEC:PROOF:UPPER:NONLINEAR}

This section is devoted to the proof of Theorem~\ref{THM:UPPER:NONLINEAR}.
The main reasoning is given in Section~\ref{SSEC:THM:UPPER:NONLINEAR} but some tedious calculations for this proof are deferred to Section~\ref{APP:BOUNDS:ETERMS}.
Sections~\ref{SSEC:CONC} and \ref{SSEC:MOMENTS} contain auxiliary results used in Section~\ref{APP:BOUNDS:ETERMS}.

\subsection{Proof of Theorem~\ref{THM:UPPER:NONLINEAR}}\label{SSEC:THM:UPPER:NONLINEAR}

As in the proof of the Corollary~\ref{COR:UPPER:LINEAR}, we note that it is sufficient to prove the result for $p \leq r$ and one can deduce the result for $p>r$ as in the proof of this corollary.

We consider the upper bound $\E \lVert \festnl - f\rVert_r^r \leq 3^{r-1} (\E \lVert A \rVert_r^r + \E \lVert B \rVert_r^r + \lVert C \rVert_r^r)$
where
\begin{align*}
    A &= \sum_{k \in \Z} (\alphahat_{\lowreslev k} - \alpha_{\lowreslev k})\phi_{\lowreslev k}, \quad B =  \sum_{j=\lowreslev}^{\uppreslev} \sum_{k \in \Z} (\betatilde_{j k} - \beta_{j k})\psi_{j k}, \quad \text{and}\\
    C &=  \sum_{k \in \Z} \alpha_{\uppreslev k}\phi_{\uppreslev k} - f.
\end{align*}
We consider the risk bounds for $\E \Vert A \Vert_r^r$, $\E \Vert B \Vert_r^r$, and $\Vert C \Vert_r^r$ separately.

\medskip

\paragraph{\bfseries Upper bound for the term $\E \Vert A \Vert_r^r$:} \quad
Putting $\alphaprime_{\lowreslev k} = \frac{1}{n} \sum_{i=1}^n \phi_{\lowreslev k}(X_i)$ it holds
\begin{align*}
        \E \Vert A \Vert_r^r \leq 2^{r-1} \E \big \Vert \sum_{k \in \Z} (\alphaprime_{\lowreslev k} - \alpha_{\lowreslev k}) \phi_{\lowreslev k} \big \Vert_r^r + 2^{r-1} \E \big \Vert \sum_k \bigg( \frac{\sigma_{\lowreslev - 1}}{n} \sum_{i=1}^n W_{i,\lowreslev - 1,k} \bigg) \phi_{\lowreslev k} \big \Vert_r^r.
\end{align*}
The first term on the right-hand side is bounded by the compact support assumption on $\phi$ and using Lemma~1 from \citep{donoho1996density} as in the non-private case (see \citep{donoho1996density}, p.~522):
\begin{equation*}
    2^{r-1} \E \big \Vert \sum_{k \in \Z} (\alphaprime_{\lowreslev k} - \alpha_{\lowreslev k}) \phi_{\lowreslev k} \big \Vert_r^r \leq C(r)2^{\lowreslev (r/2 - 1)} \sum_k \E \vert \alphaprime_{\lowreslev k} - \alpha_{\lowreslev k} \vert^r \leq C(r) \bigg( \frac{2^{\lowreslev}}{n} \bigg)^{r/2}.
\end{equation*}
Concerning the second term, first, by Fubini's theorem
\begin{align*}
    \E \big\Vert \sum_k \bigg( \frac{\sigma_{\lowreslev - 1}}{n} \sum_{i=1}^n W_{i,\lowreslev-1,k} \bigg) \phi_{\lowreslev k} \big \Vert_r^r &= \int \E \vert \sum_k \bigg( \frac{\sigma_{\lowreslev - 1}}{n} \sum_{i=1}^n W_{i,\lowreslev-1,k} \bigg) \phi_{\lowreslev k}(x) \vert^r \dd x,
\end{align*}
and the integrand on the right-hand side can be bounded as follows: for $r > 2$,
\begin{align*}
    \E \lvert \sum_k \bigg( \frac{\sigma_{\lowreslev - 1}}{n} \sum_{i=1}^n W_{i,\lowreslev-1,k} \bigg) \phi_{\lowreslev k}(x) \rvert^r &\leq \frac{C(r)}{n^r} \bigg[ \sigma_{\lowreslev - 1}^r \sum_k \lvert \phi_{\lowreslev k}(x) \rvert^r \sum_{i=1}^n \E \lvert W_{i,\lowreslev-1,k} \rvert^r\\
    &\hspace{4em}+ \bigg( \sigma_{\lowreslev - 1}^2 \sum_k \lvert \phi_{j_0k}(x) \rvert^2 \sum_{i=1}^n \E \lvert W_{i,\lowreslev-1,k} \rvert^2 \bigg)^{r/2} \bigg]\\
    &\hspace{-5em}= \frac{C(r)}{n^r} \bigg[ \sigma_{\lowreslev - 1}^r \sum_k \lvert \phi_{\lowreslev k}(x) \rvert^r n r! +  \bigg(2 n \sigma_{\lowreslev - 1}^2 \sum_k \lvert \phi_{\lowreslev k}(x) \rvert^2  \bigg)^{r/2}  \bigg]\\
    &\hspace{-5em}\lesssim \frac{1}{n^r} \bigg[ 2^{r\lowreslev} \cdot \frac{n}{\alpha^r} + \bigg( 2^{2\lowreslev} n \alpha^{-2} \bigg)^{r/2} \bigg]\\
    &\hspace{-5em}= \frac{2^{r\lowreslev}}{n^{r-1}\alpha^r} + \bigg( \frac{2^{2\lowreslev}}{n\alpha^2} \bigg)^{r/2},
\end{align*}
whereas for $r \leq 2$,
\begin{equation*}
  \E \vert \sum_k \bigg( \frac{\sigma_{\lowreslev - 1}}{n} \sum_{i=1}^n W_{i,\lowreslev-1,k} \bigg) \phi_{\lowreslev k}(x) \vert^r \lesssim \bigg( \frac{2^{2\lowreslev}}{n\alpha^2} \bigg)^{r/2}.
\end{equation*}
Thus, altogether,
\begin{equation*}
    \E \lVert A \rVert_r^r \lesssim \bigg( \frac{2^{\lowreslev}}{n} \bigg)^{r/2} + \bigg( \frac{2^{2\lowreslev}}{n\alpha^2} \bigg)^{r/2}.
\end{equation*}
Hence, for our choice of $\lowreslev$ and grant to $s < N+1$ from Assumption~\ref{COND:wavelets}, we obtain
\begin{align*}
    \E \lVert A \rVert_r^r &\lesssim \Bigg( \frac{n^{\frac{1}{2(N+1)+1}}}{n} \Bigg)^{r/2} + \Bigg( \frac{(n\privpar^2)^{\frac{2}{2(N+1)+2}}}{n\alpha^2} \Bigg)^{r/2}\\
    &= n^{-\frac{r(N+1)}{2(N+1)+1}} + (n\privpar^2)^{-\frac{r(N+1)}{2(N+1)+2}}\\
    &\leq n^{-\frac{rs}{2s+1}} + (n\privpar^2)^{-\frac{rs}{2s+2}}\\
    &\lesssim n^{-\frac{rs}{2s+1}} \vee  (n\privpar^2)^{-\frac{rs}{2s+2}} \vee \left( \frac{n}{\log n} \right)^{- \frac{r(s-1/p+1/r)}{2(s-1/p)+1}} \vee \left( \frac{n\privpar^2}{\log (n\privpar^2)} \right)^{- \frac{r(s-1/p+1/r)}{2(s-1/p)+2}},
\end{align*}
and the bound on the right-hand side is the claimed rate.

\medskip

\paragraph{\bfseries Upper bound for the term $\E \lVert B \rVert_r^r$:} 
\quad
We consider the sets 
\begin{align*}
    &\Bhat_j = \{ k : \vert \betahat_{jk} \vert > K\threshold \}, & &\Shat_j = \Bhat_j^\complement,\\
    &B_j = \{ k : \vert \beta_{jk} \vert > (K/2)\threshold \}, & &S_j = B_j^\complement,\\
    &\Bprime_j = \{ k : \vert \beta_{jk} \vert > 2Kt_{j,n,\alpha} \}, & &\Sprime_j = (\Bprime_j)^\complement,
\end{align*}
and the decomposition
\begin{align*}
    B &= \sum_{j=\lowreslev}^{\uppreslev} \sum_k (\betahat_{jk} - \beta_{jk}) \psi_{jk} \big[ \1_{ \Bhat_j \cap S_j }(k) + \1_{ \Bhat_j \cap B_j }(k) \big]\\
    &\hspace{1em}- \sum_{j=\lowreslev}^{\uppreslev} \sum_k \beta_{jk}\psi_{jk} \big[ \1_{ \Shat_j \cap \Bprime_j }(k) + \1_{ \Shat_j \cap \Sprime_j }(k) \big]\\
    &\eqdef (e_{bs} + e_{bb}) - (e_{sb} + e_{ss}).
\end{align*}
Appropriate bounds for the four terms $e_{bs}, e_{bb}, e_{sb}, e_{ss}$ are derived in Appendix~\ref{APP:BOUNDS:ETERMS}.

\medskip

\paragraph{\bfseries Upper bound for the term $\lVert C \rVert_r^r$:} \quad
In the case we consider, $p \leq r$, we use the embedding $\Besov{s}{p}{q} \subset \Besov{s^\prime}{r}{\infty}$, where we recall that $s^\prime = s - \frac 1p + \frac 1r$. Then, it holds
\begin{equation*}
    \lVert \sum_{k \in \Z} \alpha_{\uppreslev k}\phi_{\uppreslev k} - f \rVert_r^r \leq C \lVert f \rVert_{spq}^r \cdot 2^{-\uppreslev \sprime r}.
\end{equation*}
Moreover, with our choice of $j_1$,
\begin{align*}
    2^{-\uppreslev \sprime r} &\leq 2^{-\uppreslev^\prime \sprime r} + 2^{-\uppreslev^{\pprime} \sprime r}\\
    &\lesssim \bigg( \frac{n}{\log n} \bigg)^{-
    \frac{r\sprime}{2(s-1/p)+1}} + \bigg(\frac{n \alpha^2}{\log(n \alpha^2)}\bigg)^{-\frac{r\sprime}{2(s-1/p)+2}},
\end{align*}
and the sum on the right-hand side is bounded from above by the claimed rate.

\subsection{Bounds for the terms \texorpdfstring{$e_{bs}$}{ebs}, \texorpdfstring{$e_{bb}$}{ebb}, \texorpdfstring{$e_{sb}$}{esb}, and \texorpdfstring{$e_{ss}$}{ess}}\label{APP:BOUNDS:ETERMS}

Consider the event $A_{jk}$ defined via $A_{jk} = \{ \lvert \betahat_{jk} - \beta_{jk} \rvert > K/2 \cdot \threshold \}$.
The concentration inequality \eqref{EQ:CONC:INEQ:AJK} for this event as well as the bound \eqref{EQ:NORM:BOUND:RANDOM:FCT} will be used frequently in the sequel without further reference.
In the following, we bound the terms $\E \lVert e_{bs} \rVert_r^r$, $\E \lVert e_{bb} \rVert_r^r$, $\E \lVert e_{sb} \rVert_r^r$, and $\E \lVert e_{ss} \rVert_r^r$ separately.

\subsubsection{Bound for $e_{bs}$}

By the Cauchy-Schwarz inequality and the fact that $\Bhat_j \cap S_j \subset A_{jk}$,
\begin{align*}
    \E \lVert e_{bs} \rVert_r^r &\lesssim \sum_{j=\lowreslev}^{\uppreslev} 2^{j(r/2-1)} \sum_{k \in \Nc_j} \E [\lvert \betahat_{jk} - \beta_{jk}\rvert^r \1_{\Bhat_j \cap S_j}(k)]\\
    &\leq \sum_{j=\lowreslev}^{\uppreslev} 2^{j(r/2-1)} \sum_{k \in \Nc_j} \big( \E [\lvert \betahat_{jk} - \beta_{jk}\rvert^{2r}] \big)^{1/2} \cdot \P( \lvert \betahat_{jk} - \beta_{jk} \rvert \geq K/2 \cdot t_{j,n,\alpha})^{1/2}\\
    &\leq \sum_{j=j_0}^{\uppreslev} 2^{j(r/2-1)} \sum_{k \in \Nc_j} (n^{-r/2} \vee j^{\nu r/2} 2^{jr/2} (n\privpar^2)^{-r/2}) \cdot 2^{-\gamma j/2}\\
    &\leq \sum_{j=j_0}^{\uppreslev} 2^{jr/2} 2^{-\gamma j/2} (n^{-r/2} \vee j^{\nu r/2} 2^{jr/2} (n\privpar^2)^{-r/2})\\
    &\leq n^{-r/2} \sum_{j=j_0}^{\uppreslev} 2^{jr/2} 2^{-\gamma j/2} + (n\privpar^2)^{-r/2} \uppreslev^{\nu r/2} \sum_{j=j_0}^{\uppreslev} 2^{jr} 2^{-\gamma j/2}
\end{align*}
and this term is bounded from above by the claimed rate provided that $\gamma \geq 2r$.

\subsubsection{Bound for $e_{sb}$}
Using the relation $\Shat_j \cap \Bprime_j \subset A_{jk}$, we obtain
\begin{align*}
    \E \lVert e_{sb} \rVert_r^r &\lesssim \sum_{j=j_0}^{j_1} 2^{j(r/2 - 1)} \sum_k \lvert \beta_{jk} \rvert^r \cdot \E [\1_{ \Shat_j \cap B_j^\prime }(k)]\\
    &\lesssim \sum_{j=j_0}^{j_1} 2^{j(r/2 - 1)} \sum_{k \in \Nc_j} \lvert \beta_{jk} \rvert^r \cdot  \P(\lvert \betahat_{jk} - \beta_{jk} \rvert \geq K \cdot t_{j,n,\alpha})\\
    &\lesssim  \sum_{j=j_0}^{j_1} 2^{j (\frac r2 - 1 - \gamma)} \|\beta_{j \cdot}\|_r^r .
\end{align*}
In the considered case $p \leq r $, we exploit the embedding $\Besov{s}{p}{q} \subseteq \Besov{\sprime}{r}{q}$ with $\sprime = s -\frac 1p + \frac 1r$ to get the bound
\[ \|\beta_{j \cdot}\|_r \lesssim 2 ^{-j(s'+\frac 12 - \frac 1r)}. \]
Hence,
\[ \E \lVert e_{sb} \rVert_r^r \lesssim \sum_{j=j_0}^{j_1} 2^{j(\frac r2 - 1 - \gamma)} 2^{-jr ( \sprime + \frac 12 - \frac 1r )} = \sum_{j=j_0}^{j_1} 2^{-j(\gamma + r\sprime)} \lesssim 2^{-\lowreslev(\gamma + r\sprime)}\]
by the definition of $\lowreslev$.
Noting that
\begin{align*}
    2^{-\lowreslev(\gamma + r\sprime)} &\lesssim (n\privpar^2)^{-\frac{\gamma+r\sprime}{2(N+1)+2}} \vee n^{-\frac{\gamma + r\sprime}{2(N+1)+1}}\\
    &\leq (n\privpar^2)^{-\frac{rs}{2s+2}} \vee n^{-\frac{rs}{2s+1}}
\end{align*}
provided that $\gamma$ is large enough ($\gamma \geq r(N+1)$ is sufficient), shows that $\E \lVert e_{sb} \rVert_r^r$ is at most of the same order as the claimed rate.

\subsubsection{Bound for $e_{bb}$}

Put $\threshold^\prime= \gamma j^{\nu + 1/2}n^{-1/2}$ and $\threshold^{\pprime} = \gamma j^{\nu + 1/2}(n\privpar^2)^{-1/2}2^{j/2}$.
Note that $\threshold = \threshold^\prime \vee \threshold^{\pprime}$.
For any $\rho \geq 0$, it holds
\begin{align}
    \E \lVert e_{bb} \rVert_r^r &\lesssim \sum_{j=j_0}^{j_1} 2^{j(r/2-1)} \sum_k \E [\lvert \betahat_{jk} - \beta_{jk} \rvert^r  \1_{ \Bhat_j \cap B_j }(k)]\nonumber \\
    &\lesssim \sum_{j=j_0}^{j_1} 2^{j(r/2-1)} \sum_k  (n^{-r/2} \vee j^{\nu r/2} 2^{j r/2} (n\privpar^2)^{-r/2}) \1_{B_j}(k) \nonumber\\
    &\lesssim \sum_{j=j_0}^{j_1} 2^{j(r/2-1)} \sum_k  n^{-r/2} \1_{B_j}(k) \notag\\
    &\hspace{1em}+ \sum_{j=j_0}^{j_1} 2^{j(r/2-1)} \sum_k j^{\nu r/2} 2^{j r/2} \cdot (n\privpar^2)^{-r/2} \1_{B_j}(k) \nonumber\\
    &\lesssim \sum_{j=\lowreslev}^{\uppreslev} 2^{j(r/2 -1)} (\threshold^{\prime})^{r} \cdot \sum_k |\beta_{jk}|^\rho (\threshold^{\prime})^{-\rho} \notag\\
    &\hspace{1em}+ \sum_{j=\lowreslev}^{\uppreslev} 2^{j(r/2 -1)} (\threshold^{\pprime})^{r} \sum_k |\beta_{jk}|^\rho (\threshold^{\pprime})^{-\rho} \nonumber\\
    &\lesssim \underbrace{\sum_{j=j_0}^{j_1} 2^{j(r/2 -1)} (\threshold^{\prime})^{r-\rho} \sum_k \vert \beta_{jk}\vert^\rho}_{\eqdef S_1} + \underbrace{\sum_{j=j_0}^{j_1} 2^{j(r/2 -1)} (\threshold^{\pprime})^{r-\rho} \cdot \sum_k \vert \beta_{jk}\vert^\rho}_{\eqdef S_2}. \label{EQ:bound}
\end{align}
As this argument shows, one can even choose distinct values of $\rho$ for different $j$, which will be used in the following calculations. 
Note that
\begin{equation*}
    \sum_k \lvert \beta_{jk} \rvert^p \lesssim 2^{-jp(s+1/2-1/p)},
\end{equation*}
and, if $\rho \leq p$, by H{\"o}lder's inequality 
\begin{equation*}
    \sum_{k} \lvert \beta_{jk} \rvert^\rho \leq 2^{j(1-\rho/p)} \big( \sum_k \lvert \beta_{jk} \rvert^p \big)^{\rho/p} \leq 2^{j(1-\rho/p)} 2^{-j\rho(s+1/2-1/p)}=2^{-j\rho(s+1/2-1/\rho)}.
\end{equation*}
In the sequel, we consider three different cases corresponding to the three regimes in the statement of Theorem~\ref{THM:UPPER:NONLINEAR}.
\bigskip

\paragraph{\emph{1. Case}: $p > r/(s+1)$}

\begin{itemize}[topsep=1em, itemsep=1em, leftmargin=1em]
    \item Bound for $S_1$: Set $q_1=r/(2s+1)$ and define $\kappa_1 \in \N$ such that \[2^{\kappa_1(r/2-p/2-sp)} \asymp n^{-\frac{p-q_1}{2}}.\]

Choosing $\rho < q_1 \leq p$  for the indices $j \in \llbr \lowreslev, \kappa \rrbr$, we obtain (note that $s+1/2=r/(2q_1)$)
\begin{align*}
    \sum_{j=j_0}^{\kappa_1} 2^{j(r/2 -1)} (\threshold^{\prime})^{r-\rho} \sum_k \vert \beta_{jk}\vert^\rho &\lesssim \uppreslev^{(r-\rho)(\nu + 1/2)}n^{-(r-\rho)/2} \sum_{j=j_0}^{\kappa_1} 2^{j(r/2 -1)} \sum_k \vert \beta_{jk}\vert^{\rho}\\
&\lesssim \uppreslev^{(r-\rho)(\nu + 1/2)}n^{-(r-\rho)/2} \sum_{j=\lowreslev}^{\kappa_1} 2^{j(r/2 - \rho(s+1/2))}\\
    &\lesssim (\log n)^C n^{-(r-\rho)/2} 2^{\kappa_1 (r/2-\frac{r\rho}{2q_1})}\\
    &\asymp (\log n)^C n^{(q_1-r)/2}\\
    &= (\log n)^C n^{-\frac{rs}{2s+1}}.
\end{align*}

Choosing $\rho = p$ for indices $j \in \llbr \kappa_1 +1, \uppreslev\rrbr$, we obtain
\begin{align*}
    \sum_{j=\kappa_1+1}^{\uppreslev} 2^{j(r/2 -1)} (\threshold^{\prime})^{r-\rho} \sum_k \vert \beta_{jk}\vert^\rho &\lesssim \uppreslev^{(r-p)(\nu + 1/2)}n^{-(r-p)/2} \sum_{j=\kappa_1+1}^{\uppreslev} 2^{j(r/2-sp - p/2)}\\
    &\lesssim \uppreslev^{(r-p)(\nu + 1/2)}n^{-(r-p)/2} 2^{\kappa_1(r/2-sp - p/2)}\\
    &\lesssim \uppreslev^{(r-p)(\nu + 1/2)}n^{-(r-q_1)/2}\\
    &\asymp (\log n)^C \cdot n^{-\frac{rs}{2s+1}}.
\end{align*}
\item Bound for $S_2$: Set $q_2=r/(s+1)$ and define $\kappa_2 \in \N$ such that \[2^{\kappa_2(r-p-sp)} \asymp (n\privpar^2)^{-\frac{p-q_2}{2}}.\]

Choosing $\rho < q_2 \leq p$ for the indices $j \in \llbr \lowreslev,\kappa_2\rrbr$, we obtain (note that $s+1=r/q_2$)
\begin{align*}
    \sum_{j=j_0}^{\kappa_2} 2^{j(r/2 -1)} (\threshold^{\pprime})^{r-\rho} \cdot \sum_k \vert \beta_{jk}\vert^\rho &\lesssim  \uppreslev^{(r-\rho)(\nu + 1/2)} (n\privpar^2)^{-(r-\rho)/2} \sum_{j=\lowreslev}^{\kappa_2} 2^{j(r - \rho(s+1))}\\
    &\lesssim (\log n)^C (n\privpar^2)^{-(r-\rho)/2} 2^{\kappa_2(r - \frac{r\rho}{q_2})}\\
    &\asymp (\log n)^C (n\privpar^2)^{(q_2-r)/2}\\
    &= (\log n)^C (n\privpar^2)^{-\frac{rs}{2s+2}}.
\end{align*}

Choosing $\rho = p$ for indices $j \in \llbr \kappa_2 +1, \uppreslev\rrbr$, we obtain
\begin{align*}
    \sum_{j=\kappa_2+1}^{\uppreslev} 2^{j(r/2 -1)} (\threshold^{\pprime})^{r-\rho} \sum_k \vert \beta_{jk}\vert^\rho &\lesssim \uppreslev^{(r-p)(\nu + 1/2)}(n\privpar^2)^{-(r-p)/2} \sum_{j=\kappa_2+1}^{\uppreslev} 2^{j(r -sp - p)}\\
    &\lesssim \uppreslev^{(r-p)(\nu + 1/2)}(n\privpar^2)^{-(r-p)/2} 2^{\kappa_2(r-sp - p)}\\
    &\lesssim (\log n)^C \cdot (n\privpar^2)^{(q_2-r)/2}\\
    &=(\log n)^C \cdot (n\privpar^2)^{-\frac{rs}{2s+2}}.
\end{align*}
\end{itemize}

\paragraph{\emph{2. Case}: $p \in (r/(2s+1),r/(s+1)]$}

\begin{itemize}[topsep=1em, itemsep=1em, leftmargin=1em]
    \item Bound for $S_1$: The sum $S_1$ can be dealt with as in the first case, since the choices of $q_1$ and $\kappa_1$ from that case are still legitimated for $p \in (r/(2s+1),r/(s+1)]$. 
    \item Bound for $S_2$: In order to bound $S_2$ in the second case, define $q_2$ and $\kappa_2$ via the relations
\[ q_2 = r\cdot\frac{1-1/r}{s-1/p+1} \qquad \text{and} \qquad 2^{\kappa_2} \asymp (n\privpar^2)^{\frac{q_2}{2(r-1)}}. \]
To deal with the sum over $j \in \llbr \lowreslev,\kappa_2 \rrbr$, we take $\rho = p$ and obtain
\begin{align*}
    \sum_{j=j_0}^{\kappa_2} 2^{j(r/2 -1)} (\threshold^{\pprime})^{r-\rho} \sum_k \vert \beta_{jk}\vert^\rho &\lesssim \uppreslev^{(\nu + 1/2)(r-p)} (n\privpar^2)^{(r-p)/2} \sum_{j=\lowreslev}^{\kappa_2} 2^{j(r-sp-p)}\\
    &\lesssim (\log n)^C (n\privpar^2)^{-(r-p)/2} \sum_{j=\lowreslev}^{\kappa_2} 2^{j(r-1)(1-p/q_2)}\\
    &\lesssim (\log n)^C (n\privpar^2)^{-(r-p)/2} 2^{\kappa_2(r-1)(1-p/q_2)}\\
    &\asymp (\log n)^C (n\privpar^2)^{(q_2-r)/2}\\
    &= (\log n)^C (n\privpar^2)^{-\frac{r\sprime}{2(s-1/p)+2}}
\end{align*}

For the sum over indices $j \in \llbr \kappa_2 + 1,\uppreslev\rrbr$, we choose $\rho > q_2 > p$, and obtain by monotony of $\ell^{\omega}$-norms in $\omega$, and putting $s^{\pprime} = s+1/2-1/p$ that
\begin{align*}
    \sum_{j=\kappa_2 + 1}^{\uppreslev} 2^{j(r/2 -1)} (\threshold^{\pprime})^{r-\rho} \sum_k \vert \beta_{jk}\vert^\rho &\lesssim \uppreslev^{(\nu + 1/2)(r-\rho)} (n\privpar^2)^{-\frac{r-\rho}{2}} \sum_{j=\kappa_2 + 1}^{\uppreslev} 2^{j(r-\rho/2-1-\rho s^{\pprime})}\\
    &\lesssim (\log n)^C (n\privpar^2)^{-\frac{r-\rho}{2}} \sum_{j=\kappa_2 + 1}^{\uppreslev} 2^{j(r-1-\rho/2 -\rho s^{\pprime})}\\
    &=(\log n)^C (n\privpar^2)^{-\frac{r-\rho}{2}} \sum_{j=\kappa_2 + 1}^{\uppreslev} 2^{j(r-1)(1-\rho/q_2)}\\
    &\lesssim (\log n)^C (n\privpar^2)^{-\frac{r-\rho}{2}} 2^{\kappa_2(r-1)(1-\rho/q_2)}\\
    &\asymp (\log n)^C (n\privpar^2)^{\frac{q_2 -r}{2}}\\
    &= (\log n)^C (n\privpar^2)^{-\frac{r\sprime}{2(s-1/p)+2}}.
\end{align*}
\end{itemize}

\paragraph{\emph{3. Case}: $p \leq r/(2s+1)$}

\begin{itemize}[topsep=1em, itemsep=1em, leftmargin=1em]
    \item Bound for $S_1$: Put
\[ q_1 = r \cdot \frac{1/2 - 1/r}{s+1/2-1/p}, \]
and choose $\kappa_1 \in \N$ such that
\[ 2^{\kappa_1} \asymp n^{\frac{1}{2}\frac{q_1}{r/2 - 1}}. \]
Then, taking $\rho = p$ for the indices $j \in \llbr \lowreslev,\kappa_1\rrbr$ in the first sum in \eqref{EQ:bound}, we obtain
\begin{align*}
    \sum_{j=j_0}^{\kappa_1} 2^{j(r/2 -1)} (\threshold^{\prime})^{r-\rho} \sum_k \vert \beta_{jk}\vert^\rho &\leq \uppreslev^{(\nu + 1/2)(r-p)} n^{-(r-p)/2} \sum_{j=j_0}^{\kappa_1} 2^{j(r/2 - sp - p/2)}\\
    &\lesssim (\log n)^C n^{-(r-p)/2} \sum_{j=j_0}^{\kappa_1} 2^{j(r/2-1)(1-p/q_1)}\\
    &\lesssim (\log n)^C n^{-(r-p)/2} 2^{\kappa_1(r/2-1)(1-p/q_1)}\\
    &= (\log n)^C n^{\frac{q_1-r}{2}}\\
    &= (\log n)^C n^{-\frac{r\sprime}{2(s-1/p)+1}}.
\end{align*}
For the sum over indices $j \in \llbr \kappa_1 + 1,\uppreslev\rrbr$, we choose $\rho > q_1 > p$, and obtain by monotony of $\ell^{\omega}$-norms in $\omega$ and putting $s^{\pprime} = s+1/2-1/p$ that
\begin{align*}
    \sum_{j=\kappa_1 + 1}^{\uppreslev} 2^{j(r/2 -1)} (\threshold^{\prime})^{r-\rho} \sum_k \vert \beta_{jk}\vert^\rho &\lesssim \uppreslev^{(\nu + 1/2)(r-\rho)} n^{-\frac{r-\rho}{2}} \sum_{j=\kappa_1 + 1}^{\uppreslev} 2^{j(r/2-1-\rho s^{\pprime})}\\
    &\lesssim (\log n)^C n^{-\frac{r-\rho}{2}} \sum_{j=\kappa_1 + 1}^{\uppreslev} 2^{j(r/2-1-\rho s^{\pprime})}\\
    &=(\log n)^C n^{-\frac{r-\rho}{2}} \sum_{j=\kappa_1 + 1}^{\uppreslev} 2^{j(r/2-1)(1-\rho/q_1)}\\
    &\lesssim (\log n)^C n^{-\frac{r-\rho}{2}} 2^{\kappa_1(r/2-1)(1-\rho/q_1)}\\
    &\asymp (\log n)^C n^{\frac{q_1 -r}{2}}\\
    &= (\log n)^C n^{-\frac{r\sprime}{2(s-1/p)+1}}.
\end{align*}

    \item Bound for $S_2$: $S_2$ can be dealt with exactly as in the second case.
\end{itemize}

\subsubsection{Bound for $e_{ss}$}

For any $0 \leq \rho \leq r$
\begin{align*}
    \E \lVert e_{ss} \rVert_r^r &\lesssim \sum_{j=j_0}^{j_1} 2^{j(\frac r2 -1)} \sum_k \lvert \beta_{jk} \rvert^r \E [\1_{\Shat_j \cap \Sprime_j}(k)] \\
    &\lesssim \sum_{j=j_0}^{j_1} 2^{j(\frac r2-1)} ((\threshold^\prime)^{r-\rho} \vee (\threshold^{\pprime})^{r-\rho}) \sum_k \lvert \beta_{jk} \rvert^\rho .
\end{align*}
This term can be bounded from above by the right-hand side of \eqref{EQ:bound}, and we conclude in the same way as for the term $e_{bb}$.

\subsection{A concentration inequality for the \texorpdfstring{$\betahat_{jk}$}{estimators of detail coefficients}}\label{SSEC:CONC}

For our proof, we need concentration inequalities for the events \begin{equation*}
    A_{jk} \defeq \left\{ \vert \betahat_{jk} - \beta_{jk} \vert \geq (K/2)\frac{j^{\nu+1/2}}{\sqrt{n}} \left(1 \vee \frac{2^{j/2}}{\alpha}\right) \right\}
\end{equation*}
for $K>0$, where $j \in \llbr \lowreslev, \uppreslev \rrbr$ and $k \in \Nc_j$.
Let recall the two-sided Bernstein's inequality (cf.~\cite{boucheron2013concentration} Theorem~2.10).

\begin{theorem}
Let $Y_1,\ldots,Y_n$ be independent real valued random variables. Assume that there exist some positive numbers $v$ and $c$ such that
\begin{equation}
\label{BernsteinHypothese1}
\sum_{i=1}^n\mathbb{E}[Y_i^2]\leqslant v,
\end{equation}
and for all integers $m\geqslant 3$
\begin{equation}
\label{BernsteinHypothese2}
\sum_{i=1}^n \mathbb{E}[\vert Y_i\vert^m]\leqslant \frac{m!}{2}vc^{m-2}.
\end{equation}
Let $S=\sum_{i=1}^n(Y_i-\mathbb{E}[Y_i])$, then for every positive $x$
\begin{equation}
\label{TwoSidedBernsteinsInequality}
\mathbb{P}\left[ \vert S\vert\geqslant \sqrt{2vx}+cx\right]\leqslant 2\exp(-x).
\end{equation}
\end{theorem}

Using this inequality, we can prove the following result.

\begin{proposition}
For all $j \in \llbr \lowreslev, \uppreslev \rrbr$ satisfying $j\leq n$, for all $k\in\Nc_j$, and for all $\gamma\geqslant 1$ we have
\begin{equation}
\label{EQ:CONC:INEQ:AJK}
\mathbb{P}\left(\vert \betahat_{jk}-\beta_{jk}\vert \geq 4(\bar{c} + \sigma)\gamma \frac{j^{\nu+1/2}}{\sqrt{n}} \left(1 \vee \frac{2^{j/2}}{\alpha}\right) \right)\leqslant 2^{-\gamma j},
\end{equation}
where $\bar c$ is an upper bound for $\sup_{f \in \DBesov{s}{p}{q}(L,T)} \|f\|_\infty$ and $\sigma=4c_A\Vert\psi\Vert_\infty (2\nu-1)/(\nu-1)$ appears in the privacy mechanism \eqref{EQ:DEF:ZIJK}.
\end{proposition}

\begin{remark}
By Equation~(15) in \cite{donoho1996density}, the choice $\bar c = L$ is admissible for $f \in \DBesov{s}{p}{q}(L,T)$.
\end{remark}

\begin{proof}
We will apply Bernstein's inequality to the random variables $\{Z_{ijk}\}_{i=1,\ldots,n}$.
Using that $\psi_{jk}(X_i)$ and $W_{ijk}$ are independent and that $\mathbb{E}[W_{ijk}]=0$, we get for all $i\in\llbracket 1,n\rrbracket$
\begin{align*}
\mathbb{E}[Z_{ijk}^2]&=\mathbb{E}[\psi_{jk}(X_i)^2]+\sigma_j^2\mathbb{E}[W_{ijk}^2]+2\sigma_j\mathbb{E}[\psi_{jk}(X_i)W_{ijk}]\\
&=\mathbb{E}[\psi_{jk}(X_i)^2]+\sigma_j^2\mathbb{E}[W_{ijk}^2]+2\sigma_j\mathbb{E}[\psi_{jk}(X_i)]\mathbb{E}[W_{ijk}]\\
&=\mathbb{E}[\psi_{jk}(X_i)^2]+\sigma_j^2\mathbb{E}[W_{ijk}^2]\\
&\leqslant \bar c +2\sigma_j^2\\
&\leqslant 2(\bar c+\sigma_j)^2,
\end{align*}
where $\bar c>0$ depends on $L$ is such that $\|f\|_\infty \leq \bar c$ for all $f $ in $\Besov{s}{p}{q}(L)$ with $s>\frac 1p$.
Let $m\geqslant 3$ be an integer. Using again that $\psi_{jk}(X_i)$ and $W_{ijk}$ are independent we get for all $i\in\llbracket 1,n\rrbracket$
\begin{align*}
\mathbb{E}[\vert Z_{ijk}\vert^m]&\leqslant \mathbb{E}[(\vert\psi_{jk}(X_i)\vert+\sigma_j\vert W_{ijk}\vert)^m]\\
&=\mathbb{E}\left[\sum_{l=0}^m \binom{m}{l}\sigma_j^l\vert W_{ijk}\vert^l\vert \psi_{jk}(X_i)\vert^{m-l} \right]\\
&=\sum_{l=0}^m \binom{m}{l}\sigma_j^l\mathbb{E}\left[\vert W_{ijk}\vert^l\right]\mathbb{E}\left[\vert \psi_{jk}(X_i)\vert^{m-l}\right] \\
&=\sum_{l=0}^m \binom{m}{l}\sigma_j^l \mathbb{E}\left[\vert \psi_{jk}(X_i)\vert^{m-l}\right]l! \\
&\leqslant m! \sum_{l=0}^m \binom{m}{l}\sigma_j^l (\bar c)^{m-l}\\
&=m!(\bar c +\sigma_j)^m.
\end{align*}
Conditions (\ref{BernsteinHypothese1}) and (\ref{BernsteinHypothese2}) are thus satisfied with $v=2n(\bar c +\sigma_j)^2$ and $c=\bar c+\sigma_j$, and according to Bernstein's inequality (\ref{TwoSidedBernsteinsInequality}) we have for all $x>0$
$$\mathbb{P}\left(\vert \betahat_{jk}-\beta_{jk}\vert \geqslant (\bar c +\sigma_j)\left(2\sqrt{\frac{x}{n}}+\frac{x}{n} \right) \right)\leqslant 2\exp(-x).$$
Note that we have for all $j \in \llbr \lowreslev, \uppreslev \rrbr$,
$$
\bar c + \sigma_j = \bar{c} + \sigma j^{\nu} \frac{2^{j/2}}{\alpha}\leqslant (\bar{c} + \sigma )j^{\nu} \left( 1 \vee \frac{2^{j/2}}{\alpha} \right),
$$
where $\sigma=4c_A\Vert\psi\Vert_\infty (2\nu-1)/(\nu-1)$ appears in the definition of $\sigma_j$ in \eqref{EQ:DEF:ZIJK}.
Take $x=Cj$, $C>0$ and note that $2\sqrt{Cj/n}+Cj/n\leqslant (2\sqrt{C}+C)\sqrt{j/n}$ if $j\leqslant n$.
Consequently, we get for all $C>0$, for all $j \in \llbr \lowreslev, \uppreslev \rrbr$ satisfying $j\leqslant n$ and for all $k\in\Nc_j$,
$$
\mathbb{P}\left(\vert \betahat_{jk}-\beta_{jk}\vert \geqslant (\bar{c} + \sigma )(C+2\sqrt{C}) \frac{j^{\nu+1/2}}{\sqrt{n}} \left(1 \vee \frac{2^{j/2}}{\alpha}\right) \right)\leqslant 2\exp(-Cj).
$$
Then, it suffices to take $C=2\ln(2)\gamma$ to obtain (\ref{EQ:CONC:INEQ:AJK}).
\end{proof}

\subsection{Moment bounds and norm inequalities}\label{SSEC:MOMENTS}

Consider an arbitrary random function
\begin{equation*}
    \ghat = \sum_{j=\lowreslev}^{\uppreslev} \sum_k \ghat_{jk} \psi_{jk}.
\end{equation*}
Putting
\begin{equation*}S(\iota) = \sum_{j=\lowreslev}^{\uppreslev} 2^{j\iota} \leq \begin{cases} c_\gamma 2^{\max (\uppreslev \iota, \lowreslev \iota)} & \text{ if }\iota \neq 0,\\
    \uppreslev - \lowreslev, & \text{ if } \iota = 0,\end{cases}
\end{equation*}
it has been shown in \cite{donoho1996density} that for arbitrary $v \in \R$ and $u=r/(r-2)$ it holds
\begin{align*}
     \E \Vert \ghat \Vert_r^r \leq \begin{cases} C^r \sum_{j=\lowreslev}^{\uppreslev} 2^{j(r/2-1)} \sum_k \E \lvert \ghat_{jk}\rvert^r, & \text{ if } 1 \leq r \leq 2,\\
     C^r S(uv)^{(r/2-1)_+} \sum_{j=\lowreslev}^{\uppreslev} 2^{j(r/2 - 1 - vr/2)} \sum_{k \in \Z} \E \lvert \ghat_{jk} \rvert^r, & \text{ if } r > 2.\end{cases}\end{align*}
As in \cite{donoho1996density}, adopting the formal convention $S^0=1$, it suffices to consider the second inequality for all $r \geq 1$ (setting $v=0$ for the case $r \leq 2$).
Thus, for any $r \geq 1$,
\begin{equation}\label{EQ:NORM:BOUND:RANDOM:FCT}
    \E \Vert \ghat \Vert_r^r \lesssim \sum_{j=\lowreslev}^{\uppreslev} 2^{j(r/2-1)} \sum_k \E \lvert \ghat_{jk}\rvert^r. 
\end{equation}

Consider again the decomposition $\betahat_{jk} = \beta_{jk}^\prime + \frac{\sigma_j}{n}  \sum_{i=1}^n W_{ijk}$.
We have, for any $m \geq 1$,
\begin{equation*}
    \E \vert \betahat_{jk} - \beta_{jk} \vert^m \leq 2^{m-1} \E \vert \beta_{jk}^\prime - \beta_{jk} \vert^m+ 2^{m-1} \E \vert \frac{\sigma_j}{n}  \sum_{i=1}^n W_{ijk} \vert^m.
\end{equation*}
In \citep{donoho1996density}, p.~520, Equation~(16) it is shown that
\begin{equation}\label{EQ:MOMENT:BOUND:BETAPRIME}
    \E \vert \beta_{jk}^\prime - \beta_{jk} \vert^m \leq cn^{-m/2}
\end{equation}
provided that $2^j \leq n$ with a constant $c$ depending only on $s$, $p$, $q$, $L$, $\Vert \psi \Vert_m$ and $m$.
In addition, by Rosenthal's inequality, it can be shown for any $m \geq 1$ that
\begin{equation}\label{EQ:MOMENT:BOUND:WIJK}
    \E \bigg\lvert \frac{\sigma_j}{n}  \sum_{i=1}^n W_{ijk} \bigg\rvert^m \lesssim j^{\nu m/2} 2^{j m/2} (n\privpar^2)^{-m/2}.
\end{equation}
Combining~\eqref{EQ:MOMENT:BOUND:BETAPRIME} and \eqref{EQ:MOMENT:BOUND:WIJK} yields
\begin{equation}\E \vert \betahat_{jk} - \beta_{jk} \vert^m \lesssim  n^{-m/2} \vee j^{\nu m/2} 2^{j m/2} (n\privpar^2)^{-m/2}.
\end{equation}

\section*{Acknowledgements}

The authors gratefully acknowledge financial support from GENES.
Cristina Butucea and Martin Kroll also gratefully acknowledge financial support from the French National Research Agency (ANR) under the grant Labex Ecodec (ANR-11-LABEX-0047). 
\printbibliography

\end{document}